\documentclass[10pt]{amsart}
\usepackage{amsmath}
\usepackage{amsxtra}
\usepackage{amscd}
\usepackage{amsthm}
\usepackage{amsfonts}
\usepackage{amssymb}
\usepackage{eucal}
\usepackage[all]{xy}
\usepackage{graphicx}
\usepackage{stmaryrd}

\usepackage[usenames]{color}

\newtheorem{cor}[subsubsection]{Corollary}
\newtheorem{lem}[subsubsection]{Lemma}
\newtheorem{prop}[subsubsection]{Proposition}

\newtheorem{thm}[subsubsection]{Theorem}

\theoremstyle{definition}
\newtheorem{defn}[subsubsection]{Definition}
\newtheorem{example}[subsubsection]{Example}

\newcommand{\iso}{\buildrel{\sim}\over{\longrightarrow}}

\theoremstyle{definition}

\theoremstyle{remark}
\newtheorem{rem}[subsubsection]{Remark}

\newcommand{\thmref}[1]{Theorem~\ref{#1}}
\newcommand{\secref}[1]{Sect.~\ref{#1}}
\newcommand{\lemref}[1]{Lemma~\ref{#1}}
\newcommand{\propref}[1]{Proposition~\ref{#1}}
\newcommand{\corref}[1]{Corollary~\ref{#1}}

\newcommand{\remref}[1]{Remark~\ref{#1}}

\numberwithin{equation}{section}

\newcommand{\nc}{\newcommand}
\nc{\renc}{\renewcommand}
\nc{\ssec}{\subsection}
\nc{\sssec}{\subsubsection}
\nc{\on}{\operatorname}

\nc\ol{\overline}
\nc\wt{\widetilde}
\nc\tboxtimes{\wt{\boxtimes}}
\nc{\alp}{\alpha}

\newcommand{\limto}{{\displaystyle\lim_{\longrightarrow}}}
\newcommand{\rightlim}{\mathop{\limto}}

\newcommand{\leftlim}{\mathop{\displaystyle\lim_{\longleftarrow}}}
\newcommand{\limfromn}{\leftlim\limits_{\raise3pt\hbox{$n$}}}
\newcommand{\limton}{\rightlim\limits_{\raise3pt\hbox{$n$}}}

\nc{\subscheme}{{``sub"\-scheme }}
\nc{\subschemes}{{``sub"schemes }}
\nc{\subspaces}{{``sub"spaces }}
\nc{\subspace}{{``sub"space }}

\newcommand{\mono}{\hookrightarrow}
\newcommand{\epi}{\twoheadrightarrow}

\nc{\Zp}{{Z_{\infty}^+}}
\nc{\Rn}{{R_n}}

\nc{\BIG}{{\Phi}}
\nc{\ZZ}{{\mathbb Z}}
\nc{\NN}{{\mathbb N}}
\nc{\OO}{{\mathbb O}}
\renc{\SS}{{\mathbb S}}
\nc{\DD}{{\mathbb D}}
\nc{\GG}{{\mathbb G}}

\nc{\Fq}{{\mathbb F}_q}
\nc{\Fqb}{\ol{{\mathbb F}_q}}
\nc{\Ql}{\ol{{\mathbb Q}_\ell}}
\nc{\id}{\text{id}}
\nc\X{\mathcal X}

\nc{\bl}{\on{bl}}
\nc{\st}{\on{st}}
\nc{\der}{\on{der}}
\nc{\Eq}{\on{Eq}}
\nc{\ET}{\on{et}}
\nc{\Et}{\on{Et}}
\nc{\ev}{\on{ev}}
\nc{\fin}{\on{fin}}
\nc{\Hom}{\on{Hom}}
\nc{\QC}{\on{QC}}
\nc{\Lie}{\on{Lie}}
\nc{\Loc}{\on{Loc}}
\nc{\Pic}{\on{Pic}}
\nc{\Bun}{\on{Bun}}
\nc{\IC}{\on{IC}}
\nc{\Aut}{\on{Aut}}
\nc{\pr}{\on{pr}}
\nc{\rk}{\on{rk}}
\nc{\Sh}{\on{Sh}}
\nc{\Perv}{\on{Perv}}
\nc{\pos}{{\on{pos}}}
\nc{\Conv}{\on{Conv}}
\nc{\sep}{\on{sep}}
\nc{\Sph}{\on{Sph}}
\nc{\Sym}{\on{Sym}}
\nc{\BunBb}{\overline{\Bun}_B}
\nc{\BunBbm}{\overline{\Bun}_{B^-}}
\nc{\BunBbel}{\overline{\Bun}_{B,el}}
\nc{\BunBbmel}{\overline{\Bun}_{B^-,el}}
\nc{\Buno}{\overset{o}{\Bun}}
\nc{\BunPb}{{\overline{\Bun}_P}}
\nc{\BunBM}{\Bun_{B(M)}}
\nc{\BunBMb}{\overline{\Bun}_{B(M)}}
\nc{\BunPbw}{{\widetilde{\Bun}_P}}
\nc{\BunBP}{\widetilde{\Bun}_{B,P}}
\nc{\GUb}{\overline{G/U}}
\nc{\GUPb}{\overline{G/U(P)}}

\nc{\Hhom}{\underline{\on{Hom}}}
\nc\syminfty{\on{Sym}^{\infty}}
\nc\lal{\ol{\lambda}}
\nc\xl{\ol{x}}
\nc\thl{\ol{\theta}}
\nc\nul{\ol{\nu}}
\nc\mul{\ol{\mu}}
\nc\Sum\Sigma
\nc{\oX}{\overset{o}{X}{}}
\nc{\hl}{\overset{\leftarrow}h{}}
\nc{\hr}{\overset{\rightarrow}h{}}
\nc{\M}{{\mathcal M}}
\nc{\N}{{\mathcal N}}
\nc{\F}{{\mathcal F}}
\nc{\D}{{\mathcal D}}
\nc{\Q}{{\mathcal Q}}
\nc{\Y}{{\mathcal Y}}
\nc{\G}{{\mathcal G}}
\nc{\E}{{\mathcal E}}
\nc{\CalC}{{\mathcal C}}
\nc\Dh{\widehat{\D}}

\nc{\C}{{\mathcal C}}
\nc{\K}{{\mathcal K}}
\renewcommand{\H}{{\mathcal H}}

\nc{\T}{{\mathcal T}}
\nc{\V}{{\mathcal V}}
\renc{\P}{{\mathcal P}}
\nc{\A}{{\mathcal A}}
\nc{\B}{{\mathcal B}}
\nc{\U}{{\mathcal U}}

\nc{\Gr}{{\on{Gr}}}

\nc{\frn}{{\check{\mathfrak u}(P)}}
\nc{\p}{\mathfrak p}
\nc{\q}{\mathfrak q}
\nc\f{{\mathfrak f}}

\nc{\qo}{{\mathfrak q}}
\nc{\po}{{\mathfrak p}}
\nc{\s}{{\mathfrak s}}
\nc\w{\text{w}}

\nc\mathi\iota
\nc\Spec{\on{Spec}}
\nc\Spf{\on{Spf}}
\nc\Mod{\on{Mod}}
\nc{\tw}{\widetilde{\mathfrak t}}
\nc{\pw}{\widetilde{\mathfrak p}}
\nc{\qw}{\widetilde{\mathfrak q}}
\nc{\jw}{\widetilde j}

\nc{\grb}{\overline{\Gr}}
\nc{\I}{\mathcal I}

\nc{\lambdach}{{\check\lambda}}
\nc{\Lambdach}{{\check\Lambda}{}}
\nc{\much}{{\check\mu}}
\nc{\omegach}{{\check\omega}}
\nc{\nuch}{{\check\nu}}
\nc{\etach}{{\check\eta}}
\nc{\alphach}{{\check\alpha}}
\nc{\betach}{{\check\beta}}
\nc{\rhoch}{{\check\rho}}
\nc{\ch}{{\check h}}

\nc{\Hb}{\overline{\H}}

\nc{\BA}{{\mathbb{A}}}
\nc{\BB}{{\mathbb{B}}}
\nc{\BC}{{\mathbb{C}}}
\nc{\BG}{{\mathbb{G}}}
\nc{\BM}{{\mathbb{M}}}
\nc{\BD}{{\mathbb{D}}}
\nc{\BN}{{\mathbb{N}}}
\nc{\BP}{{\mathbb{P}}}
\nc{\BQ}{{\mathbb{Q}}}
\nc{\BR}{{\mathbb{R}}}
\nc{\BX}{{\mathbb{X}}}
\nc{\BXp}{\mathbb{X}^+}
\nc{\BXm}{\mathbb{X}^-}
\nc{\BXpm}{\mathbb{X}^{\pm}}

\nc{\BY}{{\mathbb{Y}}}
\nc{\BZ}{{\mathbb{Z}}}
\nc{\BS}{{\mathbb{S}}}

\nc{\CA}{{\mathcal{A}}}
\nc{\CB}{{\mathcal{B}}}

\nc{\CE}{{\mathcal{E}}}
\nc{\CF}{{\mathcal{F}}}
\nc{\CG}{{\mathcal{G}}}

\nc{\CL}{{\mathcal{L}}}
\nc{\CC}{{\mathcal{C}}}
\nc{\CM}{{\mathcal{M}}}
\nc{\CN}{{\mathcal{N}}}
\nc{\CK}{{\mathcal{K}}}
\nc{\CO}{{\mathcal{O}}}
\nc{\CP}{{\mathcal{P}}}
\nc{\CQ}{{\mathcal{Q}}}
\nc{\CR}{{\mathcal{R}}}
\nc{\CS}{{\mathcal{S}}}
\nc{\oCS}{\overset{\circ}{\mathcal{S}}}
\nc{\CT}{{\mathcal{T}}}
\nc{\CU}{{\mathcal{U}}}
\nc{\CV}{{\mathcal{V}}}
\nc{\CW}{{\mathcal{W}}}
\nc{\CX}{{\mathcal{X}}}
\nc{\CY}{{\mathcal{Y}}}
\nc{\CZ}{{\mathcal{Z}}}
\nc{\CI}{{\mathcal{I}}}
\nc{\CJ}{{\mathcal{J}}}

\nc{\csM}{{\check{\mathcal A}}{}}
\nc{\oM}{{\overset{\circ}{\mathcal M}}{}}
\nc{\obM}{{\overset{\circ}{\mathbf M}}{}}
\nc{\oCA}{{\overset{\circ}{\mathcal A}}{}}
\nc{\obA}{{\overset{\circ}{\mathbf A}}{}}
\nc{\ooM}{{\overset{\circ}{M}}{}}
\nc{\osM}{{\overset{\circ}{\mathsf M}}{}}
\nc{\vM}{{\overset{\bullet}{\mathcal M}}{}}
\nc{\nM}{{\underset{\bullet}{\mathcal M}}{}}
\nc{\oD}{{\overset{\circ}{\mathcal D}}{}}
\nc{\obD}{{\overset{\circ}{\mathbf D}}{}}
\nc{\oA}{{\overset{\circ}{\mathbb A}}{}}
\nc{\op}{{\overset{\bullet}{\mathbf p}}{}}
\nc{\cp}{{\overset{\circ}{\mathbf p}}{}}
\nc{\oU}{{\overset{\bullet}{\mathcal U}}{}}
\nc{\oZ}{{\overset{\circ}{\mathcal Z}}{}}
\nc{\ofZ}{{\overset{\circ}{\mathfrak Z}}{}}
\nc{\oF}{{\overset{\circ}{\fF}}}

\nc{\fa}{{\mathfrak{a}}}
\nc{\fb}{{\mathfrak{b}}}
\nc{\fd}{{\mathfrak{d}}}
\nc{\fg}{{\mathfrak{g}}}
\nc{\fgl}{{\mathfrak{gl}}}
\nc{\fh}{{\mathfrak{h}}}
\nc{\fj}{{\mathfrak{j}}}
\nc{\fl}{{\mathfrak{l}}}
\nc{\fm}{{\mathfrak{m}}}
\nc{\fn}{{\mathfrak{n}}}
\nc{\fu}{{\mathfrak{u}}}
\nc{\fp}{{\mathfrak{p}}}
\nc{\fr}{{\mathfrak{r}}}
\nc{\fs}{{\mathfrak{s}}}
\nc{\ft}{{\mathfrak{t}}}
\nc{\fsl}{{\mathfrak{sl}}}
\nc{\hsl}{{\widehat{\mathfrak{sl}}}}
\nc{\hgl}{{\widehat{\mathfrak{gl}}}}
\nc{\hg}{{\widehat{\mathfrak{g}}}}
\nc{\chg}{{\widehat{\mathfrak{g}}}{}^\vee}
\nc{\hn}{{\widehat{\mathfrak{n}}}}
\nc{\chn}{{\widehat{\mathfrak{n}}}{}^\vee}

\nc{\fA}{{\mathfrak{A}}}
\nc{\fB}{{\mathfrak{B}}}
\nc{\fD}{{\mathfrak{D}}}
\nc{\fE}{{\mathfrak{E}}}
\nc{\fF}{{\mathfrak{F}}}
\nc{\fG}{{\mathfrak{G}}}
\nc{\fK}{{\mathfrak{K}}}
\nc{\fL}{{\mathfrak{L}}}
\nc{\fM}{{\mathfrak{M}}}
\nc{\fN}{{\mathfrak{N}}}
\nc{\fP}{{\mathfrak{P}}}
\nc{\fU}{{\mathfrak{U}}}
\nc{\fV}{{\mathfrak{V}}}
\nc{\fZ}{{\mathfrak{Z}}}

\nc{\bb}{{\mathbf{b}}}
\nc{\bc}{{\mathbf{c}}}
\nc{\bd}{{\mathbf{d}}}
\nc{\be}{{\mathbf{e}}}
\nc{\bj}{{\mathbf{j}}}
\nc{\bn}{{\mathbf{n}}}
\nc{\bp}{{\mathbf{p}}}
\nc{\bq}{{\mathbf{q}}}
\nc{\bu}{{\mathbf{u}}}
\nc{\bv}{{\mathbf{v}}}
\nc{\bx}{{\mathbf{x}}}
\nc{\bs}{{\mathbf{s}}}
\nc{\by}{{\mathbf{y}}}
\nc{\bw}{{\mathbf{w}}}
\nc{\bA}{{\mathbf{A}}}
\nc{\bK}{{\mathbf{K}}}
\nc{\bB}{{\mathbf{B}}}
\nc{\bC}{{\mathbf{C}}}
\nc{\bG}{{\mathbf{G}}}
\nc{\bD}{{\mathbf{D}}}
\nc{\bH}{{\mathbf{H}}}
\nc{\bM}{{\mathbf{M}}}
\nc{\bN}{{\mathbf{N}}}
\nc{\bV}{{\mathbf{V}}}
\nc{\bW}{{\mathbf{W}}}
\nc{\bX}{{\mathbf{X}}}
\nc{\bZ}{{\mathbf{Z}}}
\nc{\bS}{{\mathbf{S}}}

\nc{\sA}{{\mathsf{A}}}
\nc{\sB}{{\mathsf{B}}}
\nc{\sC}{{\mathsf{C}}}
\nc{\sD}{{\mathsf{D}}}
\nc{\sG}{{\mathsf{G}}}
\nc{\sF}{{\mathsf{F}}}
\nc{\sK}{{\mathsf{K}}}
\nc{\sM}{{\mathsf{M}}}
\nc{\sO}{{\mathsf{O}}}
\nc{\sW}{{\mathsf{W}}}
\nc{\sQ}{{\mathsf{Q}}}
\nc{\sP}{{\mathsf{P}}}
\nc{\sZ}{{\mathsf{Z}}}
\nc{\sfp}{{\mathsf{p}}}
\nc{\sfq}{{\mathsf{q}}}
\nc{\sfr}{{\mathsf{r}}}
\nc{\osfr}{\overset{\circ}{\mathsf{r}}}
\nc{\sr}{{\mathsf{r}}}
\nc{\sfi}{{\mathsf{i}}}
\nc{\sfj}{{\mathsf{j}}}
\nc{\sk}{{\mathsf{k}}}
\nc{\sg}{{\mathsf{g}}}
\nc{\sff}{{\mathsf{f}}}
\nc{\sfb}{{\mathsf{b}}}
\nc{\sfc}{{\mathsf{c}}}
\nc{\sd}{{\mathsf{d}}}

\nc{\BK}{{\bar{K}}}

\nc{\tA}{{\widetilde{\mathbf{A}}}}
\nc{\tB}{{\widetilde{\mathcal{B}}}}
\nc{\tg}{{\widetilde{\mathfrak{g}}}}
\nc{\tG}{{\widetilde{G}}}
\nc{\TM}{{\widetilde{\mathbb{M}}}{}}
\nc{\tO}{{\widetilde{\mathsf{O}}}{}}
\nc{\tU}{{\widetilde{\mathfrak{U}}}{}}
\nc{\TZ}{{\tilde{Z}}}
\nc{\tx}{{\tilde{x}}}
\nc{\tbv}{{\tilde{\bv}}}
\nc{\tfP}{{\widetilde{\mathfrak{P}}}{}}
\nc{\tz}{{\tilde{\zeta}}}
\nc{\tmu}{{\tilde{\mu}}}

\nc{\urho}{\underline{\rho}}
\nc{\uB}{\underline{B}}
\nc{\uC}{{\underline{\mathbb{C}}}}
\nc{\ui}{\underline{i}}
\nc{\uj}{\underline{j}}
\nc{\ofP}{{\overline{\mathfrak{P}}}}
\nc{\oB}{{\overline{\mathcal{B}}}}
\nc{\og}{{\overline{\mathfrak{g}}}}
\nc{\oI}{{\overline{I}}}

\nc{\eps}{\varepsilon}
\nc{\hrho}{{\hat{\rho}}}

\nc{\one}{{\mathbf{1}}}
\nc{\two}{{\mathbf{t}}}

\nc{\Rep}{{\mathop{\operatorname{\rm Rep}}}}
\nc{\Tot}{{\mathop{\operatorname{\rm Tot}}}}
\nc{\Ker}{{\mathop{\operatorname{\rm Ker}}}}
\nc{\Coker}{{\mathop{\operatorname{\rm Coker}}}}
\nc{\im}{{\mathop{\operatorname{\rm Im}}}}
\nc{\Hilb}{{\mathop{\operatorname{\rm Hilb}}}}
\nc{\End}{{\mathop{\operatorname{\rm End}}}}
\nc{\Ext}{{\mathop{\operatorname{\rm Ext}}}}
\nc{\CHom}{{\mathop{\operatorname{{\mathcal{H}}\it om}}}}
\nc{\GL}{{\mathop{\operatorname{\rm GL}}}}
\nc{\gr}{{\mathop{\operatorname{\rm gr}}}}
\nc{\Id}{{\mathop{\operatorname{\rm Id}}}}
\nc{\de}{{\mathop{\operatorname{\rm def}}}}
\nc{\length}{{\mathop{\operatorname{\rm length}}}}
\nc{\supp}{{\mathop{\operatorname{\rm supp}}}}

\nc{\Cliff}{{\mathsf{Cliff}}}
\nc{\Fl}{\on{Fl}}
\nc{\Fib}{{\mathsf{Fib}}}
\nc{\Coh}{{\mathsf{Coh}}}
\nc{\FCoh}{{\mathsf{FCoh}}}

\nc{\reg}{{\text{\rm reg}}}

\nc{\cplus}{{\mathbf{C}_+}}
\nc{\cminus}{{\mathbf{C}_-}}
\nc{\cthree}{{\mathbf{C}_\bullet}}
\nc{\Qbar}{{\bar{Q}}}
\nc\Eis{\on{Eis}}
\nc\Eisb{\ol\Eis{}}
\nc\wh{\widehat}
\nc{\Def}{\on{Def_{\check{\fb}}(E)}}
\nc{\barZ}{\overline{Z}{}}
\nc{\barbarZ}{\overline{\barZ}{}}
\nc{\barpi}{\overline\pi}
\nc{\barbarpi}{\overline\barpi}
\nc{\barpip}{\overline\pi{}^+}
\nc{\barpim}{\overline\pi{}^-}

\nc{\fq}{\mathfrak q}

\nc{\sfqb}{\ol{\sfq}{}}
\nc{\sfpb}{\ol{\sfp}{}}

\nc{\hattimes}{\wh\otimes}

\nc{\bh}{{\bar{h}}}
\nc{\bOmega}{{\overline{\Omega(\check \fn)}}}

\nc{\seq}[1]{\stackrel{#1}{\sim}}

\nc{\cT}{{\check{T}}}
\nc{\cG}{{\check{G}}}
\nc{\cM}{{\check{M}}}
\nc{\cB}{{\check{B}}}

\nc{\ct}{{\check{\mathfrak t}}}
\nc{\cg}{{\check{\fg}}}
\nc{\cb}{{\check{\fb}}}
\nc{\cn}{{\check{\fn}}}

\nc{\cLambda}{{\check\Lambda}}

\nc{\cla}{{\check\lambda}}
\nc{\cmu}{{\check\mu}}
\nc{\cnu}{{\check\nu}}
\nc{\ceta}{{\check\eta}}

\nc{\DefbE}{{\on{Def}_{\cB}(E_\cT)}}

\nc{\imathb}{{\ol{\imath}}}
\nc{\Dmod}{\on{D-mod}}
\nc{\Maps}{\on{Maps}}
\nc{\MMaps}{\on{\mathbf{Maps}}}
\nc{\GMaps}{{\MMaps^{\BG_m}}}
\nc{\gMaps}{{\Maps^{\BG_m}}}
\nc{\Vect}{\on{Vect}}
\nc{\CMaps}{\mathcal Maps}
\nc{\sotimes}{\overset{!}\otimes}
\nc{\dr}{\on{dR}}
\nc{\oBX}{\overset{\circ}\BX}

\nc{\red}{\on{red}}

\begin{document}

\title{On a theorem of Braden}

\author{V.~Drinfeld and D.~Gaitsgory}

\dedicatory{Dedicated to E.~Dynkin} 

\date{\today}

\begin{abstract}
We give a new proof of Braden's theorem (\cite{Br}) about \emph{hyperbolic restrictions}
of constructible sheaves/D-modules. The main geometric ingredient in the proof is
a 1-parameter family that degenerates a given scheme $Z$ equipped with a $\BG_m$-action
to the product of the attractor and repeller loci.
\end{abstract}

\maketitle

\tableofcontents

\section*{Introduction}

\ssec{The setting for Braden's theorem}

\sssec{}
Given a scheme $Z$ (or algebraic space) of finite type over a field $k$ of characteristic 0, 
let $\Dmod(Z)$ denote the DG category of D-modules on it. 

\medskip

If $f:Z_1\to Z_2$ is a morphism of such schemes one has the de Rham direct image functor 
$f_{\bullet}:\Dmod(Z_1)\to \Dmod(Z_2)$ and the $!$-pullback functor 
$f^!:\Dmod(Z_2)\to \Dmod(Z_1)$.
One also has the \emph{partially} defined functor $f^\bullet:\Dmod(Z_2)\to \Dmod(Z_1)$, left adjoint to 
$f_{\bullet}\,$.

\sssec{}

Suppose now that $Z$ is equipped with an action of the group $\BG_m\,$. Let $Z^+$ (resp., $Z^-$) denote the 
corresponding attractor (resp., repeller) locus, see Sects~\ref{ss:attr} and \ref{ss:repeller} 
for the definitions.  Let $Z^0$ denote the locus of $\BG_m$-fixed points. 

\medskip

Consider the diagram
\begin{equation} \label{e:square with arrow prev}
\xy
(30,0)*+{Z^+}="Y";
(0,-30)*+{Z^-}="Z";
(30,-30)*+{Z.}="W";
(-5,5)*+{Z^0}="U";
{\ar@{->}_{p^-} "Z";"W"};
{\ar@{->}^{p^+} "Y";"W"};
{\ar@{->}^{i^-} "U";"Z"};
{\ar@{->}_{i^+} "U";"Y"};
\endxy
\end{equation}

\medskip

Let $\Dmod(Z)^{\BG_m\on{-mon}}\subset \Dmod(Z)$ be the full subcategory consisting of $\BG_m$-monodromic\footnote{The definition of $\BG_m$-monodromic object is recalled in Subsect.~\ref{sss:mon}.} objects. 
In the context of D-modules, Braden's theorem \cite{Br} 
(inspired by a result\footnote{In \cite{GM} M.~Goresky and R.~MacPherson  work in a purely 
topological setting. They work with correspondences rather than torus actions. According to \cite[Prop. 9.2]{GM}, 
under a certain condition (which is satisfied if the correspondence comes from a $\BG_m$-action) one has 
$A_4^\bullet\iso A_5^\bullet\,$, where  $A_i^\bullet$ is defined in  \cite[Prop. 4.5]{GM}. 
This is the proptotype of Braden's theorem.} from \cite{GM}) says that 
the composed functors
$$(i^+)^\bullet\circ (p^+)^! \text{ and } (i^-)^!\circ (p^-)^\bullet, \quad \Dmod(Z)\to \Dmod(Z^0)$$
are both defined \footnote{The issue here is that in the context of D-modules
the $\bullet$-pullback functor is only partially defined.}
on objects of $\Dmod(Z)^{\BG_m\on{-mon}}$  
and we have a canonical isomorphism
\begin{equation} \label{e:Braden preview}
(i^+)^\bullet\circ (p^+)^!|_{\Dmod(Z)^{\BG_m\on{-mon}}}  \simeq (i^-)^!\circ (p^-)^\bullet|_{\Dmod(Z)^{\BG_m\on{-mon}}}\;\; .
\end{equation} 

\sssec{}

In his paper \cite{Br}, T.~Braden formulated and proved his theorem assuming that $Z$ is a normal algebraic variety.
Although his formulation is enough for practical purposes, we prefer to formulate and prove this theorem for 
\emph{algebraic spaces} of finite type over a field (without any normality or separateness conditions).  

\medskip

In this more general context, the representability of the functors defining
the attractor $Z^+$ (and other related spaces such as $\wt{Z}$ from \secref{sss:behind} below) 
is no longer obvious; it is established in \cite{Dr}. 

\ssec{Why should we care?}

Braden's theorem is hugely important in geometric representation theory. 

\sssec{}

Here is a typical application
in the context of Lusztig's theory of induction and restriction of character sheaves.

\medskip

Take $Z=G$, a connected reductive group. Let $P\subset G$ be a parabolic, and let $P^-$ be an opposite parabolic,
so that $M:=P\cap P^-$ identifies with the Levi quotient of both $P$ and $P^-$. Denote the corresponding
closed embeddings by
$$M\overset{i^+}\hookrightarrow P\overset{p^+}\hookrightarrow G \text{ and }
M\overset{i^-}\hookrightarrow P^-\overset{p^-}\hookrightarrow G.$$

\medskip

Then the claim is that we have a canonical isomorphism of functors $$\Dmod(G)^{\on{Ad}_G\on{-mon}}\to \Dmod(M)^{\on{Ad}_M\on{-mon}}$$
between the corresponding categories of $\on{Ad}$-monodromic D-modules:
\begin{equation} \label{e:restr}
(i^+)^\bullet\circ (p^+)^!\simeq (i^-)^!\circ (p^-)^\bullet.
\end{equation}

\medskip

The proof is immediate from \eqref{e:Braden preview}: the corresponding $\BG_m$-action is the adjoint action
corresponding to a co-character $\BG_m\to M$, which maps to the center of $M$ and is dominant regular with
respect to $P$. \footnote{Unfortunately, it seems that this particular proof of the isomorphism 
\eqref{e:restr}, although very simple and well-known in the folklore, does not appear in the published literature.}

\sssec{}

For other applications of Braden's theorem see 
\cite{Ach},  \cite{AC},  \cite{AM},   \cite{Bi-Br},   \cite{GH},   \cite{Ly1},   \cite{Ly2},   \cite{MV},   \cite{Nak}.

\ssec{The new proof}

The goal of this paper is to give an alternative proof of Braden's theorem. The reason for our decision
to publish it is that 

\smallskip

\noindent(a) the new proof gives another point of view on ``what Braden's theorem is really about";

\smallskip

\noindent(b) a slight modification of the new proof of Braden's theorem allows to prove a new result in the geometric theory of automorphic forms, see \cite[Thm.~1.2.5]{DrGa3}.

\medskip

Let us explain the idea of the new proof. 

\sssec{Braden's theorem as an adjunction}

Let us complete the diagram \eqref{e:square with arrow prev} to 
\begin{equation} \label{e:square with arrow prev enh}
\xy
(30,0)*+{Z^+}="Y";
(0,-30)*+{Z^-}="Z";
(30,-30)*+{Z,}="W";
(-5,5)*+{Z^0}="U";
{\ar@{->}_{p^-} "Z";"W"};
{\ar@{->}^{p^+} "Y";"W"};
{\ar@{->}^{i^-} "U";"Z"};
{\ar@{->}_{i^+} "U";"Y"};
{\ar@<-1.3ex>_{q^+} "Y";"U"};
{\ar@<1.3ex>^{q^-} "Z";"U"};
\endxy
\end{equation}
where 
$$q^+:Z^+\to Z^0 \text{ and } q^-:Z^-\to Z^0$$
are the corresponding contraction maps. 

\medskip

First, we observe that the functors $(i^+)^\bullet$ and $(i^-)^!$, when restricted to the corresponding monodromic categories,
are isomorphic to $(q^+)_\bullet$ and $(q^-)_!$, respectively. Hence, 
the isomorphism \eqref{e:Braden preview} can be 
rewritten as
\begin{equation} \label{e:Braden preview q}
(q^+)_\bullet\circ (p^+)^!|_{\Dmod(Z)^{\BG_m\on{-mon}}}  \simeq (q^-)_!\circ (p^-)^\bullet|_{\Dmod(Z)^{\BG_m\on{-mon}}}.
\end{equation} 

Next we observe that the functor $(q^-)_!\circ (p^-)^\bullet$ is the left adjoint functor of 
$(p^-)_\bullet\circ (q^-)^!$. Hence,
\emph{the isomorphism \eqref{e:Braden preview q} can be restated as the assertion that the functors
$$(q^+)_\bullet\circ (p^+)^!|_{\Dmod(Z)^{\BG_m\on{-mon}}} \text{ and } (p^-)_\bullet\circ (q^-)^!|_{\Dmod(Z)^{\BG_m\on{-mon}}}$$
form an adjoint pair.}

\sssec{The geometry behind the adjunction}   \label{sss:behind}
In turns out that the co-unit for this adjunction, i.e., the map
\begin{equation}   \label{e:counit preview}
(q^+)_\bullet\circ (p^+)^!\circ (p^-)_\bullet\circ (q^-)^!\to \on{Id}_{\Dmod(Z^0)}\, ,
\end{equation}
is easy to write down (just as in the original form of Braden's theorem, a map in
one direction is obvious). 

\medskip

The crux of the new proof consists of writing down the unit for the adjunction, i.e., the corresponding map
\begin{equation} \label{e:unit preview}
\on{Id}_{\Dmod(Z)^{\BG_m\on{-mon}}}\to (p^-)_\bullet\circ (q^-)^! \circ (q^+)_\bullet\circ (p^+)^!|_{\Dmod(Z)^{\BG_m\on{-mon}}}\, .
\end{equation}

\medskip

The map \eqref{e:unit preview} comes from a certain geometric construction described in \secref{s:deg}.
Namely, we construct a $1$-parameter ``family" 
\footnote{The quotation marks are due to the fact that this 
``family" is not flat, in general. If $Z$ is affine  then each $\wt{Z}_t$ is a closed subscheme of $Z\times Z$. 
If $Z$ is separated, then for each $t$ the map 
$\wt{Z}_t\to Z\times Z$ is a monomorphism (but not necesaarily a locally closed embedding).} 
of schemes (resp., algebraic spaces) $\wt{Z}_t$ mapping to
$Z\times Z$ (here $t\in \BA^1$) such that for $t\ne 0$ the scheme (resp., algebraic space) $\wt{Z}_t$ is the graph of the map $t:Z\iso Z$,
and $\wt{Z}_0$ is isomorphic to $Z^+\underset{Z^0}\times Z^-$.

\ssec{Other sheaf-theoretic contexts}  \label{ss:other}

This paper is written in the context of D-modules on schemes (or more generally, 
algebraic spaces of finite type) over a field $k$ of characteristic 0.

\medskip

However, Braden's theorem can be stated in other sheaf-theoretic contexts, where the role of
the DG category $\Dmod(Z)$ is played by a certain triangulated category $D(Z)$.
The two other contexts that we have in mind are as follows:

\smallskip

\noindent(i) $k$ is any field, and $D(Z)$ is the derived category of ${\mathbb Q}_\ell$-sheaves with constructible cohomologies, 

\smallskip

\noindent(ii) $k=\BC$, and $D(Z)$ is the derived category of sheaves of $R$-modules
with constructible cohomologies (where $R$ is any ring). 

\medskip

In these two contexts the new proof of Braden's theorem presented in this article goes through with the following modifications:

\medskip

First, the functors $f^\bullet$ and $f_!$ are always defined, so one should not worry about pro-categories.\footnote{Pro-categories are considered in Appendix~\ref{s:pro}.} 

\medskip

Second, the definition of the $G$-monodromic category $D(Z)^{G\on{-mon}}$ (where $G$ is any algebraic group, e.g., the group $\BG_m$) should be slightly different from the definition of $\Dmod(Z)^{G\on{-mon}}$
given in \secref{sss:mon}.

\medskip

Namely, $D(Z)^{G\on{-mon}}$ should be defined as the full subcategory of $D(Z)$ \emph{strongly generated} by the essential 
image of the pullback functor $D(Z/G)\to D(Z)$ (i.e., its objects are those
objects of $D(Z)$ that can be obtained from objects lying in the image of the above pullback
functor by a \emph{finite} iteration of the procedure of taking the cone of a morphism). 

\ssec{Some conventions and notation}

\sssec{}

In Sects. \ref{s:actions} and \ref{s:deg}
we will work over an arbitrary ground field $k$, and in Sects. \ref{s:Braden1}-\ref{s:Verifying}
we will assume that $k$ has characteristic $0$ (because we will be working with D-modules). 

\sssec{}
In this article all schemes, algebraic spaces, and stacks are assumed to be ``classical" (as opposed to derived).


\sssec{}
When working with D-modules, our conventions follow those of \cite[Sects. 5 and 6]{DrGa1}. The 
only notational difference is that for a morphism $f:Z_1\to Z_2$, we will denote the
direct image functor $\Dmod(Z_1)\to \Dmod(Z_2)$ by $f_\bullet$ (instead of $f_{\dr,*}$),
and similarly for the left adjoint, $f^\bullet$ (instead of $f^*_{\dr}$).

\sssec{}
Given an an algebraic space (or stack) $Z$  of finite type over a field $k$ of characteristic 0, we
let $\Dmod(Z)$ denote the DG category of D-modules on it. 
Our conventions regarding DG categories follow those of \cite[Sect. 0.6]{DrGa1}. 

\medskip

On the other hand, the reader may prefer to replace each time the DG category $\Dmod(Z)$ by its
homotopy category, which is a triangulated category.\footnote{However, the most natural approach to 
\emph{constructing} the triangulated category of D-modules on an algebraic stack is to construct the corresponding DG category first, as is done in \cite{DrGa1}.} Then the formulations and proofs of the main results of this article will remain valid. Moreover, once we  know that the morphism~\eqref{e:counit preview} in the triangulated setting is the co-unit of an adjunction, it follows that the same is true in the DG setting.


\medskip

\sssec{}

In Appendix \ref{s:pro} we define the notion of pro-completion $\on{Pro}(\bC)$ of a DG category $\bC$. 

\medskip

The reader who prefers to 
stay in the triangulated world, can replace it by the category
of all covariant triangulated functors from the homotopy category 
$\on{Ho}(\bC)$ to the homotopy category of complexes of $k$-vector spaces. (Note that the category of such functors is not
necessarily triangulated, but this is of no consequence for us.)

\ssec{Organization of the paper}

\sssec{}

Sects. \ref{s:actions}-\ref{s:deg} are devoted to the geometry of $\BG_m$-actions on algebraic spaces. 

\medskip

Let $Z$ be an algebraic space of finite type over the ground field $k$, equipped with a $\BG_m$-action.
In \secref{s:actions} we define the attractor $Z^+$ and the repeller $Z^-$ by
\begin{equation}  \label{e:attr&repel}
Z^+:=\GMaps(\BA^1,Z), \quad\quad Z^-:=\GMaps(\BA^1_-\, ,Z),
\end{equation}
where $\GMaps$ stands for the space of $\BG_m$-equivariant maps and $\BA^1_-:=\BP^1-\{\infty\}$ (or equivalently, $\BA^1_-$ 
is the affine line equipped with the $\BG_m$-action opposite to the usual one).
The basic facts on $Z^\pm$ are formulated in \secref{s:actions}; the proofs of the more difficult statements are given in~\cite{Dr}. 

\medskip

As was already mentioned in ~\secref{sss:behind}, in the proof of Braden's theorem we use a certain 1-parameter family of 
algebraic spaces $\wt{Z}_t$, $t\in\BA^1$. These spaces are defined and studied in 
\secref{s:deg}. The definition is formally similar to \eqref{e:attr&repel}: namely,
\[
\wt{Z}_t:=\GMaps (\BX_t\,, Z),
\]
where $\BX_t$ is the hyperbola $\tau_1\cdot \tau_2=t$ and the action of $\lambda\in\BG_m$ on 
$\BX_t$ is defined by 
$$\tilde\tau_1=\lambda\cdot \tau_1\, , \quad\quad\tilde\tau_2=\lambda^{-1}\cdot\tau_2\, .$$
Note that $\BX_0$ is the union of the two coordinate axes, which meet at the origin; accordingly, 
$\wt{Z}_0$ identifies with $Z^+\underset{Z^0}\times Z^-$ (as promised in \secref{sss:behind}).

\sssec{}

In \secref{s:Braden1} we first state Braden's theorem in its original formulation, and then reformulate
it as a statement that certain two functors are adjoint (with the specified co-unit of the adjunction). 

\medskip

In \secref{s:unit} we carry out the main step in the proof of \thmref{t:Braden adj} by constructing the
unit morphism for the adjunction. 

\medskip

The geometric input in the construction of the unit is the family 
$t\rightsquigarrow \wt{Z}_t$ mentioned above. The input from the theory of D-modules is the
\emph{specialization map}
$$\on{Sp}_\CK:\CK_1\to \CK_0\, ,$$
where $\CK$ is a $\BG_m$-monodromic object in $\Dmod(\BA^1\times \CY)$  (for any algebraic space/stack $\CY$), 
and where $\CK_1$ and $\CK_0$ are the !-restrictions of $\CK$ to $\{1\}\times \CY$ and $\{0\}\times \CY$,
respectively. The map $\on{Sp}_\CK$ is a simplified version of the specialization map that goes from nearby
cycles to the !-fiber.

\medskip

In \secref{s:Verifying} we show that the unit and co-unit indeed satisfy the adjunction property. 

\medskip

In Appendix \ref{s:pro} we define the notion of pro-completion $\on{Pro}(\bC)$ of a DG category $\bC$.

\ssec{Acknowledgements}

We thank A.~Beilinson,  T.~Braden, J.~Konarski, and A.~J.~Sommese for helpful discussions. 

\medskip 

The research of V. D. is partially supported by NSF grants DMS-1001660 and DMS-1303100. 
The research of D. G. is partially supported by NSF grant DMS-1063470. 

\section{Geometry of $\BG_m$-actions: fixed points, attractors, and repellers}  \label{s:actions}

In this section we review the theory of action of the multiplicative group $\BG_m$ on a scheme or algebraic space $Z$. 
Specifically, we are concerned with the fixed-point locus, denoted by $Z^0$, as well as the attractor/repeller spaces, 
denoted by $Z^+$ and $Z^-$, respectively.

\medskip

The main results of this section are \propref{p:Z^0closed} (which says that the fixed-point locus is closed), 
Theorem~\ref{t:attractors} (which ensures representability of attractor/repeller sets), 
and \propref{p:Cartesian} (the latter is used in the construction of the unit of the adjunction given in 
\secref{sss:defining co-unit}). 

\medskip

In the case of a scheme equipped with a locally linear $\BG_m$-action these 
results are well known (in a slightly different language).

\ssec{$k$-spaces} \label{ss:k spaces}

\sssec{}

We fix a field $k$ (of any characteristic).
By a $k$-\emph{space} (or simply  \emph{space}) we mean a contravariant functor $Z$ from the category of 
affine schemes to that of sets which is a sheaf for the fpqc topology. Instead of $Z(\Spec(R))$ we write simply 
$Z(R)$; in other words, we consider $Z$ as a covariant functor on the category of $k$-algebras.

\medskip

Note that for any scheme $S$ we have $Z(S)=\Maps (S,Z)$, where $\Maps$ stands for the set of morphisms between spaces. 
Usually we prefer to write $\Maps (S,Z)$ rather than $Z(S)$.

\medskip

We write $\on{pt}:=\Spec(k)$. 

\sssec{}

General spaces will appear only as ``intermediate" objects.
For us, the really geometric objects are \emph{algebraic spaces} over $k$. We will be using the definition 
of algebraic space from \cite{LM} (which goes back to  M.~Artin).
\footnote{In particular, quasi-separatedness is included into the definition of algebraic space. Thus
the quotient $\BA^1/\BZ$ (where the discrete group $\BZ$ acts by translations)
is \emph{not} an algebraic space.}

\medskip

Any quasi-separated scheme (in particular, any scheme of finite type) is an algebraic space.
The reader may prefer to restrict his attention to schemes, and even to separated schemes,
as this will cover most of the cases of interest to which the main the result of this paper, i.e.,
\thmref{t:braden original}, is applied.

\medskip

Note that in the definition of spaces, instead of considering affine schemes as ``test schemes", one can consider 
algebraic spaces (any fpqc sheaf on the category of affine  schemes uniquely extends to an fpqc sheaf on the category of 
algebraic spaces).

\sssec{}

A morphism of spaces $f:Z_1\to Z_2$ is said to be a \emph{monomorphism} if 
the corresponding map 
$$\Maps (S,Z_1)\to\Maps (S,Z_2)$$
is injective for any scheme $S$. In particular, this applies if $Z_1$ and $Z_2$ are algebraic spaces.
It is known that a morphism \emph{of finite type} between schemes (or algebraic spaces) is a monomorphism 
if and only if each of its geometric fibers is a reduced scheme with at most one point.

\medskip

A morphism of algebraic spaces is said to be \emph{unramified} if it has locally finite presentation and its geometric fibers are 
finite and reduced.

\ssec{The space of $\BG_m$-equivariant maps} \label{ss:GMaps}

\sssec{}

Let $Z_1$ and $Z_2$ be spaces. We define the space $\MMaps(Z_1,Z_2)$ by
$$\Maps(S,\MMaps(Z_1,Z_2)):=\Maps(S\times Z_1,Z_2)$$
(the right-hand side is easily seen to be an fpqc sheaf with respect to $S$). 

\sssec{}

Let $Z_1,Z_2$ be spaces equipped with an action of $\BG_m$. Then we define the space
$\GMaps(Z_1,Z_2)$ as follows: for any scheme $S$,
\begin{equation}
\Maps (S,\GMaps(Z_1,Z_2)):=\Maps (S\times Z_1,Z_2)^{\BG_m}
\end{equation}
(the right-hand side is again easily seen to be an fpqc sheaf with respect to $S$). 

\medskip

The action of $\BG_m$ on $Z_2$ induces a $\BG_m$-action on 
$\GMaps(Z_1,Z_2)$.

\sssec{}

Note that even if $Z_1$ and $Z_2$ are schemes, the space $\GMaps(Z_1,Z_2)$ does not have to be a scheme 
(or an algebraic space), in general. 

\ssec{The space of fixed points} \label{ss:fixed_points}

\sssec{}

Let $Z$ be a space equipped with an action of $\BG_m$. Then we set
\begin{equation}
Z^0:=\GMaps(\on{pt},Z).
\end{equation}

Note that $Z^0$ is a subspace of $Z$ because $\Maps (S,Z^0)=\Maps (S,Z)^{\BG_m}$ is a subset of
$\Maps (S,Z)$.

\begin{defn}
$Z^0$ is called \emph{the subspace of fixed points} of $Z$.
\end{defn}

\sssec{}

We have the following result:

\begin{prop}   \label{p:Z^0closed}
If $Z$ is an algebraic space (resp. scheme) of finite type then so is $Z^0$. Moreover, the morphism
$Z^0\to Z$ is a closed embedding.
\end{prop}

The assertion of the proposition is nearly tautological if $Z$ is separated. This case will suffice for most of
the cases of interest to which the main result of this paper applies. 

\medskip

The proof in general is given in \cite[Prop. 1.2.2]{Dr}. It is not difficult; the only surprise is that $Z^0\subset Z$ is closed even if $Z$ is not separated. (Explanation in characteristic zero: $Z^0$ is the subspace of zeros of the vector field on $Z$ corresponding to the $\BG_m$-action.)

\begin{example}  \label{ex:fixed-affine}
Suppose that $Z$ is an affine scheme $\Spec(A)$. A $\BG_m$-action on $Z$ is the same as a 
$\BZ$-grading on $A$. Namely, the $n$-th component of $A$ consists of 
$f\in \Gamma(Z,\CO_Z)$ such that $f(\lambda\cdot z)=\lambda^n\cdot f(z)$.

\medskip
 
It is easy to see that $Z^0=\Spec(A^0)$, where $A^0$ is the maximal graded quotient 
algebra of $A$ concentrated in degree 0 (in other words, $A^0$ is the quotient of $A$ 
by the ideal generated by homogeneous elements of non-zero degree).
\end{example}

\ssec{Attractors} \label{ss:attr}

\sssec{} 

Let $Z$ be a space equipped with an action of $\BG_m$. Then we set
\begin{equation}  \label{e:attr}
Z^+:=\GMaps(\BA^1,Z),
\end{equation}
where $\BG_m$ acts on $\BA^1$ by dilations.

\begin{defn}
$Z^+$ is called the \emph{attractor} of $Z$.
\end{defn}

\sssec{Pieces of structures on $Z^+$} \label{sss:structures}  \hfill

\medskip

\noindent(i) $\BA^1$ is a monoid with respect to multiplication. The action of $\BA^1$ on itself induces an 
$\BA^1$-action on $Z^+$, which extends the $\BG_m$-action defined in \secref{ss:GMaps}.

\medskip

\noindent(ii) Restricting a morphism $\BA^1\times S\to Z$ to $\{1\}\times S$ one gets a morphism $S\to Z$. Thus we get a 
$\BG_m$-equivariant morphism $p^+:Z^+\to Z$. 

\medskip

Note that if $Z$ is \emph{separated} (i.e., the diagonal morphism $Z\to Z\times Z$ is a closed embedding),
then $p^+:Z^+\to Z$ is a \emph{mono}morphism. To see this, it suffices to interpret
$p^+$ as the composition
\[
\GMaps (\BA^1,Z)\to \GMaps (\BG_m,Z)=Z.
\]

Thus if $Z$ is separated then $p^+$ identifies $Z^+(S)$ with the subset of those points $f:S\to Z$ for which the map $S\times \BG_m\to Z$,
defined by  $(s,t)\mapsto t\cdot f(s)$, extends to
a map $S\times \BA^1\to Z$; informally, the limit
\begin{equation}   \label{e:limit}
\underset{t\to 0}{lim}\,\, t\cdot z
\end{equation}
should exist.

\medskip

\noindent(iii) Recall that $Z^0=\GMaps (\on{pt}, Z)$. 
The $\BG_m$-equivariant maps $0:\on{pt}\to\BA^1$ and $\BA^1\to \on{pt}$ induce the maps
$$q^+:Z^+\to Z^0 \text{ and } i^+:Z^0\to Z^+,$$ such that $q^+\circ i^+=\id_{Z^0}$, and the composition $p^+\circ i^+$ 
is equal to the canonical embedding $Z^0\mono Z$.

\medskip

Note that if $Z$ is separated then for $z\in Z^+(S)\subset Z(S)$ the point $q^+(S)$ is the limit \eqref{e:limit}.

\sssec{The case of a contracting action}  \label{sss:contracting}

Let $Z$ be a separated space. Then it is clear that
if a $\BG_m$-action on $Z$ can be extended to an action of the monoid $\BA^1$ then such an extension is unique. 
In this case we will say that that the $\BG_m$-action is \emph{contracting}.

\medskip

\begin{prop} \label{p:contracting}
Let $Z$ be a separated space of finite type equipped with a $\BG_m$-action. 
The morphism $p^+:Z^+\to Z$ is an isomorphism if and only if the $\BG_m$-action on $Z$ is
contracting. 
\end{prop}

\begin{proof}

The ``only if" assertion follows from \secref{sss:structures}(i). For the ``if" assertion, we note that 
the $\BA^1$-action on $Z$ defines a morphism $g:Z\to Z^+$ such that the composition of the maps
\begin{equation}  \label{e:mono-epi}
Z\overset{g}\longrightarrow Z^+\overset{p^+}\longrightarrow Z
\end{equation}
equals $\id_Z$. Since the map $p^+$ is a monomorphism (see \secref{sss:structures}(ii)), the assertion follows.

\end{proof}

\begin{rem}  \label{r:contracting}
In \cite[Prop. 1.4.15]{Dr} it will be shown that if $Z$ is an \emph{algebraic} space of finite type, then 
the assertion of \propref{p:contracting} remains valid even if $Z$ is not separated: i.e., $p^+$
is an isomorphism if and only if the $\BG_m$-action on $Z$  can be extended to an $\BA^1$-action;
moreover, such an extension is unique.
\end{rem}

\sssec{The affine case}   \label{sss:attractors-affine}

Suppose that $Z$ is affine, i.e., $Z=\Spec(A)$, where $A$ is a $\BZ$-graded commutative algebra.
It is easy to see that in this case $Z^+$ is represented by the affine scheme $\Spec(A^+)$, where $A^+$ is the 
maximal $\BZ^{\geq 0}$-graded quotient algebra of $A$ (in other words, the quotient of $A$ by the ideal generated 
by by all homogeneous elements of $A$ of strictly negative degrees). 

\medskip

By Example~\ref{ex:fixed-affine}, $Z^0=\Spec(A^0)$, where $A^0$ is the maximal graded quotient algebra of 
$A$ (or equivalently, of $A^+$) concentrated in degree 0. Since the algebra $A^+$ is $\BZ^{\geq 0}$-graded, $A^0$ identifies with the
$0$-th graded component of $A^+$. Thus we obtain the homomorphisms $A^0\mono A^+\epi A^0$. They correspond to the morphisms 
$$Z^0\overset{\;\;q^+}\longleftarrow Z^+\overset{\;\;i^+}\longleftarrow Z^0\,.$$

\sssec{Attractors of open/closed subspaces}

We have:

\begin{lem}   \label{l:U^+}  \hfill
Let $Z$ be a space equipped with a $\BG_m$-action, and let $Y\subset Z$ 
be a $\BG_m$-stable open subspace. 

\smallskip

\noindent{\em(i)} Suppose that $Y\to Z$ is an open embedding. 
Then the subspace $Y^+\subset Z^+$ equals $(q^+)^{-1}(Y^0)$, where $q^+$ is the natural morphism $Z^+\to Z^0$.

\smallskip

\noindent{\em(ii)} Suppose that $Y\to Z$ is a closed embedding. 
Then the subspace $Y^+\subset Z^+$ equals $(p^+)^{-1}(Y)$, where $q^+$ is the natural morphism $Z^+\to Z$.

\end{lem}

\begin{proof}

Let $Y\to Z$ be an open embedding. For any test scheme $S$, we have to show that if
$$f:S\times \BA^1\to Z$$
is a $\BG_m$-equivariant morphism such that $\{0\}\times S\subset f^{-1}(Y)$ then 
$f^{-1}(Y)=S\times \BA^1$. This is clear because $f^{-1}(Y)\subset S\times \BA^1$ is open and $\BG_m$-stable.

\medskip

Let $Y\to Z$ be a closed embedding. An $S$-point of $(p^+)^{-1}(Y)$ is a  $\BG_m$-equivariant morphism 
$f:S\times \BA^1\to Y$ such that
$S\times \BG_m\subset f^{-1}(Y)$. Since $f^{-1}(Y)$ is closed in $S\times \BA^1$ this implies that
$f^{-1}(Y)=S\times \BA^1$, i.e., $f(S\times \BA^1)\subset Y$.

\end{proof}

\ssec{Representability of attractors}  \label{ss:Results_attractors} 

\sssec{}

We have the following assertion:

\begin{thm}   \label{t:attractors}
Let $Z$ be an algebraic space of finite type equipped with a $\BG_m$-action. Then

\smallskip

\noindent\emph{(i)} $Z^+$ is an algebraic space of finite type;

\smallskip

\noindent\emph{(ii)} The morphism $q^+:Z^+\to Z^0$ is affine.
\end{thm}

The proof of this theorem is given in \cite[Thm. 1.4.2]{Dr}. Here we will prove a particular case (see \secref{sss:loc lin}),
sufficient for most of the cases of interest to which the main result of this paper applies.

\medskip

Combining \thmref{t:attractors} with \propref{p:Z^0closed}, we obtain:

\begin{cor}  \hfill

\smallskip

\noindent\emph{(i)} If $Z$ is a separated algebraic space of finite type then so is $Z^+$.

\smallskip

\noindent\emph{(ii)} If $Z$ is a scheme of finite type then so is $Z^+$.
\end{cor}

\begin{proof}
Follows from  Theorem~\ref{t:attractors}(ii) because by \propref{p:Z^0closed}, $Z^0$ is a closed subspace of~$Z$.
\end{proof}
 
\sssec{The case of a locally linear action}  \label{sss:loc lin}

\begin{defn} \label{d:locally linear} 
An action of $\BG_m$ on a scheme $Z$ is said to be \emph{locally linear} if
$Z$ can be covered by open affine subschemes preserved by the $\BG_m$-action.  
\end{defn} 

\begin{rem}  \label{r:q-proj loc lin}

Suppose that $Z$ admits a $\BG_m$-equivariant locally closed embedding into
a projective space $\BP(V)$, where $\BG_m$ acts linearly on $V$.  Then the action
of $\BG_m$ is locally linear.

\medskip

For this reason, locally linear actions include most of the cases of interest that come
up in practice. 

\end{rem}

\begin{rem}   \label{r:locally linear} 
If $k$ is \emph{algebraically closed} and $Z_{\red}$ is a \emph{normal} separated\footnote{We do not know if separateness is really 
necessary in Sumihiro's theorem.} scheme of finite type over $k$, then by a theorem of H.~Sumihiro, \emph{any action of $\BG_m$ on  
$Z$ is locally linear}. (The proof of this theorem is contained in \cite{Sum} and also in  \cite[p.20-23]{KKMS} and \cite{KKLV}.)
\end{rem}

\sssec{}

Let us prove \thmref{t:attractors} in the locally linear case on a scheme. First, we note that \lemref{l:U^+}(i) reduces the assertion
to the case when $Z$ is affine. In the latter case, the assertion is manifest from \secref{sss:attractors-affine}.

\ssec{Further results on attractors}  \label{ss:further contractors}

The results of this subsection are included for completeness; they will not be used
for the proof of the main theorem of this paper.  

\medskip

We let $Z$ be an 
algebraic space of finite type equipped with a $\BG_m$-action. 

\sssec{}

We have:

\begin{prop} \label{p: p^+} \hfill

\smallskip

\noindent{\em(i)} If $Z$ is separated then $p^+:Z^+\to Z$ is a monomorphism.

\smallskip

\noindent{\em(ii)} If $Z$ is an affine scheme then $p^+:Z^+\to Z$ is a closed embedding. 

\smallskip

\noindent{\em(iii)} If $Z$ is proper then each geometric fiber of $p^+:Z^+\to Z$ is reduced and has exactly one geometric point.

\smallskip

\noindent{\em(iv)} The fiber of $p^+:Z^+\to Z$ over any geometric point of $Z^0\subset Z$ is reduced and has exactly one geometric point.

\end{prop}

\begin{proof}

Point (i) has been proved in \secref{sss:structures}(ii). Point (ii) is manifest from \secref{sss:attractors-affine}. Point (iii)
follows from point (i) and the fact that any morphism from $\BA^1-\{ 0\}$ to a proper scheme extends to the whole $\BA^1$.

\medskip

After base change, point (iv) is equivalent to the following lemma:
\begin{lem}    \label{l:constant}
If $f:\BA^1\to Z$ is a $\BG_m$-equivariant morphism such that
$f(1)\in Z^0$ then $f$ is constant. 
\end{lem}

\end{proof}

\begin{proof}[Proof of \lemref{l:constant}]
The map $\on{pt}\to Z$, corresponding to $f(1)\in Z(k)$ is a closed embedding (whether or not $Z$ is separated).
Hence, the assertion follows from \lemref{l:U^+}(ii).
\end{proof}

\begin{example}  \label{ex:P^1}
Let  $Z$ be the projective line $\BP^1$ equipped with the usual action of $\BG_m\,$. Then $p^+:Z^+\to Z$ is the
canonical morphism $\BA^1\sqcup\{\infty\}\to\BP^1$. In particular, $p^+$ is  \emph{not a locally closed embedding.}
\end{example}

\begin{example}  \label{e:P^1 glued}
Let $Z$ be the curve obtained from $\BP^1$ by gluing $0$ with $\infty$. Equip $Z$ with the $\BG_m$-action induced by 
the usual action on $\BP^1$. The map $\BP^1\to Z$ induces a map $(\BP^1)^+\to Z^+$. It is easy to see that the composed map
$$\BA^1\hookrightarrow (\BP^1)^+\to Z^+$$
is an isomorphism $\BA^1\simeq Z^+$. 
\end{example}

\begin{rem}

Suppose that the action of $\BG_m$ is locally linear. Then \propref{p: p^+}(ii) and \lemref{l:U^+} imply that
the map $p^+$ is, \emph{Zariski locally on the source}, a locally closed embedding. 

\medskip

Note, however, that is is not the case in general, as can be seen from Example \ref{e:P^1 glued}.

\end{rem}

\sssec{}

In the example of $\BP^1$, the restriction of $p^+:Z^+\to Z$ to each connected component of $Z^+$   \emph{is} a locally closed embedding. 
This turns out to be true in a surprisingly large class of situations (but there are also important examples when this is false):

\begin{thm}
Let $Z$ be a separated scheme over an algebraically closed field $k$ equipped with a $\BG_m$-action.
Then each of the following conditions ensures that the restriction of $p^+:Z^+\to Z$ to each connected component 
\footnote{Using the $\BA^1$-action on $Z^+$, it is easy to see that each connected component of $Z^+$ is the preimage of a connected component of 
$Z^0$ with respect to the map $q^+:Z^+\to Z^0\,$.} of $Z^+$  is a locally closed embedding:

\smallskip

\noindent{\em(i)} $Z$ is smooth;

\smallskip

\noindent{\em(ii)} $Z$ is normal and quasi-projective;

\smallskip

\noindent{\em(iii)} $Z$ admits a $\BG_m$-equivariant locally closed embedding into
a projective space $\BP(V)$, where $\BG_m$ acts linearly on $V$. 
\end{thm}

\medskip

Case (i) is due to A.~Bia{\l}ynicki-Birula \cite{Bia}. Case (iii) immediately follows from the easy case $Z=\BP(V)$. Case (ii) turns out to be a particular 
case of (iii) because by Theorem~1 from \cite{Sum}, if $Z$ is normal and quasi-projective then it admits a $\BG_m$-equivariant locally 
closed embedding into a projective space.

\begin{rem}
In case (i) the condition that $Z$ be a scheme (rather than an algebraic space) is essential, as shown by A.~J.~Sommese \cite{Som}. 

\medskip

In case (ii) the quasi-projectivity condition is essential, as shown by 
J.~Konarski \cite{Kon} using a method developed by J.~Jurkiewicz \cite{Ju1,Ju2}. 
In this example $Z$ is a 3-dimensional toric variety which is proper but not projective; it is constructed by drawing a certain picture on a 2-sphere, 
see the last page of \cite{Kon}.

\medskip

In case (ii) normality is clearly essential, see Example \ref{e:P^1 glued}.

\end{rem} 

\ssec{Differential properties}

The results of this subsection are included for the sake of completeness and will not be needed for the sequel.

\medskip

We let $Z$ be an algebraic space of finite type, equipped with an action of $\BG_m$. 

\sssec{}

First, we have:

\begin{lem}   \label{l:TZ0}
For any $z\in Z^0$ the tangent space\footnote{We define the tangent space by $T_zZ:=(T_z^*Z)^*$, where 
$T_z^*Z$ is the fiber of $\Omega^1_{Z/k}$ at $z$. (The equality $T_z^*Z=m_z/m_z^2$ holds 
\emph{if the residue field of $z$ is finite and separable} over $k$.)} $T_zZ^0\subset T_zZ$ equals $(T_zZ)^{\BG_m}$.
\end{lem}

\begin{proof}
We can assume that the residue field of $z$ equals $k$ (otherwise do base change). Then compute $T_zZ^0$ in terms of morphisms 
$\Spec k[\varepsilon ]/(\varepsilon^2 )\to Z^0$.
\end{proof}

\sssec{}

Next we claim:

\begin{prop} \label{p:unrami}
Let $Z$ be an algebraic space of finite type equipped with a $\BG_m$-action. Then the 
map $p^+:Z^+\to Z$ is unramified.
\end{prop}

\begin{proof}

We can assume that $k$ is algebraically closed. Then we have to check that for any 
$\zeta\in Z^+(k)$ the map of tangent spaces
\begin{equation}  \label{e:differential}
 T_{\zeta}Z^+\to T_{p^+({\zeta})}Z
\end{equation}
induced by $p^+:Z^+\to Z$ is injective. Let $f:\BA^1\to Z$ be the $\BG_m$-equivariant morphism corresponding to $\zeta$. Then 
\begin{equation}  \label{e:TZ+l}
T_{\zeta}Z^+=\Hom_{\BG_m}(f^*(\Omega^1_Z),\CO_{\BA^1}),
\end{equation}
and the map \eqref{e:differential} assigns to a $\BG_m$-equivariant morphism 
$\varphi :f^*(\Omega^1_Z)\to\CO_{\BA^1}$ the corresponding map between fibers at $1\in\BA^1$. 
So the kernel of \eqref{e:differential} consists of those
$\varphi\in\Hom_{\BG_m}(f^*(\Omega^1_Z),\CO_{\BA^1})$ for which $\varphi |_{\BA^1-\{ 0\}}=0$. 
This implies that $\varphi =0$ because $\CO_{\BA^1}$ 
has no nonzero sections supported at $0\in\BA^1$.

\end{proof}

\sssec{}

Finally, we claim:

\begin{prop}   \label{p:smoothness}
Suppose that $Z$ is smooth. Then $Z^0$ and $Z^+$  are  smooth. Moreover, the  morphism 
$q^+:Z^+\to Z^0$ is smooth.
\end{prop}

\begin{proof}
We will only prove that $q^+$ is smooth. (Smoothness of $Z^0$ can be proved similarly, and smoothness of $Z^+$ follows.)

\medskip

It suffices to check that $q^+$ is formally smooth. 
Let $R$ be a $k$-algebra and $\bar R=R/I$, where 
$I\subset R$ is an ideal with $I^2=0$. Let $\bar f :\Spec(\bar R)\times \BA^1\to Z$ be a $\BG_m$-equivariant morphism and let  
$\bar f_0:\Spec(\bar R)\to Z^0$ 
denote the restriction of $\bar f$ to 
$$\Spec(\bar R)\overset{\on{id}\times \{0\}}\hookrightarrow \Spec(\bar R)\times \BA^1\,.$$ 
Let $\varphi :\Spec(R)\to Z^0$ be a morphism extending 
$\bar f_0\,$. We have to extend  $\bar f$  to a $\BG_m$-equivariant morphism $f :\Spec(R)\times\BA^1\to Z$ so that $f_0=\varphi$, where
$f_0:=f|_{\Spec(R)}$. 

\medskip

Since $Z$ is smooth, we can find a \emph{not-necessarily equivariant} morphism 
$f :\Spec(\bar R)\times \BA^1\to Z$ extending $\bar f$ with $f_0=\varphi$. Then standard arguments show 
that the obstruction to existence of a $\BG_m$-equivariant $f$ with the required properties belongs to
$$H^1(\BG_m\, ,M), \quad 
M:=H^0(\Spec(\bar R)\times\BA^1\, ,\bar f^*(\Theta_Z)\otimes\CJ )\underset{\bar R}\otimes I,$$
where $\Theta_Z$ is the tangent bundle of $Z$ and $\CJ\subset\CO_{\Spec(R)\times\BA^1}$ 
is the ideal of the zero section. But $H^1$ of $\BG_m$ with coefficients in any $\BG_m$-module is zero.
\end{proof}

\ssec{Repellers}    \label{ss:repeller}

\sssec{}

Set $\BA^1_-:=\BP^1-\{\infty\}$; this is a monoid with respect to multiplication containing $\BG_m$ as a subgroup. 
One has an isomorphism of monoids
\begin{equation}   \label{e:inversion}
\BA^1\iso \BA^1_-\, , \quad\quad t\mapsto t^{-1}.
\end{equation}

\sssec{}

Given a space $Z$, equipped with a $\BG_m$-action, we set
\begin{equation}   \label{e:repel}
Z^-:=\GMaps(\BA^1_-,Z).
\end{equation}

\begin{defn}
$Z^-$ is called the \emph{repeller} of $Z$.
\end{defn}

\sssec{}

Just as in \secref{sss:structures} one defines an 
action of the monoid $\BA^1_-$ on $Z^-$ extending the action of $\BG_m\,$, a
$\BG_m$-equivariant morphism $p^-:Z^-\to Z$, and  $\BA^1_-$-equivariant morphisms  
$q^-:Z^-\to Z^0$ and $i^-:Z^0\to Z^-$ (where $Z^0$ is equipped with the trivial $\BA^1_-$-action).
One has $q^-\circ i^-=\id_{Z^0}\,$, and the composition $p^-\circ i^-$ is equal to the canonical embedding $Z^0\mono Z$.

\medskip

Using the isomorphism \eqref{e:inversion}, one can identify $Z^-$ with the attractor for the inverse 
action of $\BG_m$ on $Z$ (this identification is $\BG_m$-\emph{anti}-equivariant). Thus the results on attractors from 
Sects.~\ref{sss:attractors-affine} and \ref{ss:Results_attractors} imply similar results for repellers.

\medskip

In particular, if $Z$ is the spectrum of a $\BZ$-graded algebra $A$ then $Z^-$ canonically identifies with 
$\Spec (A^-)$,  where $A^-$ is the maximal $\BZ_-$-graded quotient algebra of $A$.

\ssec{Attractors and repellers}

In this subsection we let $Z$ be an algebraic space of finite type, equipped with an action of $\BG_m$. 

\sssec{}

First, we claim:

\begin{lem}  \label{l:closed} 
The morphisms $i^{\pm}:Z^0\to Z^{\pm}$ are closed embeddings.
\end{lem}

\begin{proof}  
It suffices to consider $i^+$.  By Theorem~\ref{t:attractors}(ii), the morphism $q^+:Z^+\to Z^0$ is separated. One has $q^+\circ i^+=\id_{Z^0}\,$. 
So $i^+$ is a closed embedding.
\end{proof}

\sssec{}

Now consider the fiber product $Z^+\underset{Z}\times Z^-$ formed using the maps
$p^{\pm}:Z^{\pm}\to Z$.

\begin{prop}  \label{p:Cartesian} 
The map
\begin{equation}    \label{e:open-closed}
j:=(i^+,i^-):Z^0\to Z^+\underset{Z}\times Z^-
\end{equation}
is both an open embedding and a closed one (i.e., is the embedding of a union of some connected components). 
\end{prop}

\begin{rem}
If $Z$ is affine then $j$ is an isomorphism (this immediately follows from the explicit description of $Z^{\pm}$ in the affine case, see 
Sects. \ref{sss:attractors-affine} and \ref{ss:repeller}). 

\medskip

In general, $j$ is not necessarily an isomorphism. To see this, note that by \eqref{e:attr} and \eqref{e:repel}, we have
\begin{equation}   \label{e:fiberprod}
Z^+\underset{Z}\times Z^-=\GMaps(\BP^1,Z)
\end{equation}
(where $\BP^1$ is equipped with the usual $\BG_m$-action), and a $\BG_m$-equivariant map $\BP^1\to Z$ does not have to be constant, in general.
\end{rem}

\begin{proof}[Proof of \propref{p:Cartesian}]

We will give a proof in the case when $Z$ is a scheme; the case of an arbitrary algebraic space is treated in \cite[Prop. 1.6.2]{Dr}. 

\medskip

Writing $j$ as 
\[
Z^0=Z^0\underset{Z}\times Z^0\,\overset{(i^+,i^-)}\longrightarrow\, Z^+\underset{Z}\times Z^-,
\]
and using \lemref{l:closed}, we see that $j$ is a closed embedding.

\medskip

To prove that $j$ is an open embedding, we note that the following diagram is Cartesian:

$$
\CD
Z^0  @>{\sim}>> \MMaps^{\BG_m}(\on{pt},Z)    @>>>   \MMaps(\on{pt},Z)    \\
@V{j}VV      @VVV     @VVV     \\
Z^+\underset{Z}\times Z^-  @>{\sim}>>    \MMaps^{\BG_m}(\BP^1,Z)    @>>>   \MMaps(\BP^1,Z). 
\endCD
$$

Now, the required result follows from the next (easy) lemma:

\begin{lem}
For a scheme $Z$, the map
$$Z=\MMaps(\on{pt},Z)  \to \MMaps(\BP^1,Z)$$
induced by the projection $\BP^1\to \on{pt}$ is an open embedding.
\end{lem}

\end{proof}

\begin{cor} \label{c:contractive}  \hfill

\smallskip

\noindent{\em(i)} If the map $p^+:Z^+\to Z$ is an isomorphism then so are the maps 
$Z^0\overset{i^-}\longrightarrow Z^-\overset{q^-}\longrightarrow Z^0$.

\smallskip

\noindent{\em(ii)} If the map $p^-:Z^-\to Z$ is an isomorphism then so are the maps 
$Z^0\overset{i^+}\longrightarrow Z^+\overset{q^+}\longrightarrow Z^0$.
\end{cor}

\begin{proof}
Let us prove (ii). By \propref{p:Cartesian}, the morphism $i^+:Z^0\to Z^+$ is an open embedding. 
It remains to show that any point $\zeta\in Z^+$ is contained in $i^+(Z^0)$. Set 
\[
U_{\zeta}:=\{t\in\BA^1\,|\, t\cdot\zeta\in i^+(Z^0)\}.
\]
We have to show that $1\in U_{\zeta}$. But $U_{\zeta}$ is an open $\BG_m$-stable subset of $\BA^1$ containing $0$, so $U_{\zeta}=\BA^1$.
\end{proof}

\section{Geometry of $\BG_m$-actions: the key construction}  \label{s:deg}

We keep the conventions and notation of \secref{s:actions}. The goal of this section is, given an algebraic space $Z$ equipped with
a $\BG_m$-action, to study a certain canonically defined algebraic space $\wt{Z}$, equipped with a morphism 
$\wt{Z}\to\BA^1\times Z\times Z$, such that for $t\in\BA^1-\{0\}$ the fiber $\wt{Z}_t$ equals
the graph of the map $t:Z\iso Z$, and the fiber $\wt{Z}_0\,$, corresponding to $t=0$, equals $Z^+\underset{Z^0}\times Z^-$.

\medskip

As was mentioned in \secref{sss:behind}, the space $\wt{Z}$ is the main geometric ingredient in the proof of 
\thmref{t:Braden adj}. However, the reader can skip this section now and return to it when the time comes.  

\medskip

The main points of this section are following.  In \secref{ss:tilde Z} we define the space $\wt{Z}$ and the main pieces of structure on it 
(e.g., the morphism $\wt{p}:\wt{Z}\to\BA^1\times Z\times Z$ and the action of $\BG_m\times \BG_m$ on $\wt{Z}$). In \secref{ss:repr of inter}
we address the question of representability of $\wt{Z}$. In \secref{ss:fiber products} we prove 
\propref{p:2open embeddings}, which plays a key role in Sect.~\ref{s:Verifying}. 

\ssec{A family of hyperbolas}  \label{ss:hyperb}

\sssec{}   \label{sss:family of hyperbolas}   
Set $\BX:=\BA^2=\Spec k[\tau_1,\tau_2]$. Throughout the paper equip $\BX$ with the structure of a scheme over 
$\BA^1$, defined by the map
\[
\BA^2\to\BA^1 ,\quad (\tau_1,\tau_2)\mapsto \tau_1\cdot \tau_2\, .
\]

Let $\BX_t$ denote the fiber of $\BX$ over $t\in\BA^1$; in other words, $\BX_t\subset\BA^2$ is the curve defined by the equation 
$\tau_1\cdot \tau_2=t$. If $t\ne 0$ then $\BX_t$ is a hyperbola, while $\BX_0$ is the ``coordinate cross" $\tau_1\cdot \tau_2=0$. 

\medskip

One has $\BX_0=\BX_0^+\cup\BX_0^-$, where
\begin{equation}  \label{e:X0+-}
\BX_0^+:=\{(\tau_1,\tau_2)\in\BA^2\,|\,  \tau_2=0\}\, , \quad \BX_0^-:=\{(\tau_1,\tau_2)\in\BA^2\,|\,  \tau_1=0\}\, .
\end{equation}

\sssec{The schemes $\BX_S\,$}   \label{sss:X_S}
For any scheme $S$ over  $\BA^1$ set
\begin{equation}
\BX_S:=\BX \underset{\BA^1}\times S\, .
\end{equation}
If $S=\Spec(R)$ we usually write $\BX_R$ instead of $\BX_S\,$.

\sssec{}   \label{sss:structure X}
We will need the following pieces of structure on $\BX$:

\smallskip

\noindent{(i)} The projection $\BX\to \BA^1$ admits two canonically defined sections:
\begin{equation} \label{e:two sections}
\sigma_1(t)=(1,t) \text{ and } \sigma_2(t)=(t,1).
\end{equation}

\medskip

\noindent{(ii)} 
The scheme $\BX$ carries a tautological action of the monoid $\BA^1\times \BA^1$:
$$(\lambda_1,\lambda_2)\cdot (\tau_1,\tau_2)=(\lambda_1\cdot \tau_1,\lambda_2\cdot \tau_2)\,.$$

\medskip

This action covers the action of $\BA^1\times \BA^1$ on $\BA^1$, given by the product
map $\BA^1\times \BA^1\to \BA^1$ and the tautological action of $\BA^1$ on itself.

\medskip

\noindent{(iii)} In particular, we obtain an action of $\BG_m\times \BG_m$ on $\BX$.

\medskip

This action covers the action of $\BG_m\times \BG_m$ on $\BA^1$, given by the product
map $\BG_m\times \BG_m\to \BG_m$ and the tautological action of $\BG_m$ on $\BA^1$.

\sssec{The $\BG_m$-action on $\BX_S\,$}    \label{sss:theaction}
Consider the action of the anti-diagonal copy of $\BG_m$ on the scheme $\BX$ from \secref{sss:structure X}(iii). That is,
\begin{equation}   \label{e:hyperbolic}
\lambda\cdot (\tau_1,\tau_2):=(\lambda\cdot \tau_1,\; \lambda{}^{-1}\cdot \tau_2).
\end{equation}

\medskip

This action preserves the morphism $\BX\to\BA^1$, so for any scheme $S$ over $\BA^1$ one obtains
an action of $\BG_m$ on $\BX_S\,$.

\begin{rem}  \label{r:theaction}
If $S$ is over $\BA^1-\{ 0\}$, then $\BX_S$ is $\BG_m$-equivariantly isomorphic to $S\times \BG_m$
by means of either of the maps $\sigma_1$ or $\sigma_2$.
\end{rem}

\ssec{Construction of the interpolation}   \label{ss:tilde Z}

\sssec{}

Given a space $Z$ equipped with a $\BG_m$-action, define
$\wt{Z}$ to be the following space over $\BA^1$: for any scheme $S$ over $\BA^1$ we set
$$\Maps_{\BA^1} (S, \wt{Z}):=\Maps (\BX_S,Z)^{\BG_m},$$
where $\BX_S$ is acted on by $\BG_m$ as in \secref{sss:theaction}.

\medskip

In other words, for any scheme $S$, an $S$-point of $\wt{Z}$ is a pair consisting of a morphism 
$S\to\BA^1$ and a $\BG_m$-equivariant morphism $\BX_S\to Z$.

\medskip

Note that for any $t\in\BA^1(k)$ the fiber $\wt{Z}_t$ has the following description:
\begin{equation}
\wt{Z}_t=\GMaps (\BX_t\, ,Z).
\end{equation}

\sssec{}  \label{sss:tilde p}

The sections $\sigma_1$ and $\sigma_2$ (see \secref{sss:structure X}(i)) define morphisms
$$\pi_1:\wt{Z}\to Z \text{ and } \pi_2:\wt{Z}\to Z,$$
respectively. 

\medskip

Let
\begin{equation}   \label{e:tilde p}
\wt{p}:\wt{Z}\to \BA^1\times Z\times Z
\end{equation}
denote the morphism whose first component is the tautological projection $\wt{Z}\to \BA^1$, and 
the second and the third components are $\pi_1$ and $\pi_2$, respectively.

\sssec{}   \label{sss:action of G_m^2}

Note that the action of the group $\BG_m\times \BG_m$ on $\BX$ from \secref{sss:structure X}(iii)
gives rise to an action of
$\BG_m\times \BG_m$ on $\wt{Z}$. This action has the following properties:

\smallskip

\noindent{(i)} It is compatible with the canonical map $\wt{Z}\to \BA^1$
via the multiplication map $\BG_m\times \BG_m\to \BG_m$ and the \emph{inverse} of the  
canonical action of $\BG_m$ on $\BA^1$.

\smallskip

\noindent{(ii)} It is compatible with $\pi_1:\wt{Z}\to Z$ via
the projection on the first factor $\BG_m\times \BG_m\to \BG_m$
and the initial action of $\BG_m$ on $Z$.

\smallskip

\noindent{(iii)} It is compatible with $\pi_2:\wt{Z}\to Z$ via
the projection on the second factor $\BG_m\times \BG_m\to \BG_m$
and the \emph{inverse} of the initial action of $\BG_m$ on $Z$.

\sssec{}  \label{sss:anti-diagonal}
Restricting to the \emph{anti-diagonal} copy of $\BG_m\subset \BG_m\times \BG_m$
(i.e., $\lambda\mapsto (\lambda,\lambda^{-1})$), we obtain an action
of $\BG_m$ on $\wt{Z}$. (It is the same action as the one induced by the initial action of $\BG_m$ on $Z$).
This $\BG_m$-action on $\wt{Z}$ preserves the projection $\wt{Z}\to \BA^1$. 

\medskip

Both maps $\pi_1$ and $\pi_2$ are $\BG_m$-equivariant.

\sssec{}

For $t\in \BA^1$ let 
\begin{equation}   \label{e:tilde p_t}
\wt{p}_t:\wt{Z}_t\to Z\times Z
\end{equation}
denote the morphism induced by \eqref{e:tilde p} (as before, $\wt{Z}_t$ stands for the fiber of $\wt{Z}$ over $t$).

\medskip

It is clear that $(\wt{Z}_1,\wt{p}_1)$ identifies with $(Z,\Delta_Z:Z\to Z\times Z)$.
More generally, for $t\in \BA^1-\{0\}$ the pair $(\wt{Z}_t,\wt{p}_t)$ identifies with the graph of the map $Z\to Z$ given 
by the action of $t\in \BG_m\,$. 

\medskip

More precisely, we have:

\begin{prop}   \label{p:outside 0}
The morphism \eqref{e:tilde p} induces an isomorphism between  
$$\BG_m\underset{\BA^1}\times \wt{Z}$$ 
and the graph of the action morphism $\BG_m\times Z\to Z$.
\qed
\end{prop}

\sssec{}

We are now going to construct a canonical morphism
\begin{equation} \label{e:0 fiber}
\wt{Z}_0\to Z^+\underset{Z^0}\times Z^-.
\end{equation}

Recall that $\wt{Z}_0=\GMaps (\BX_0\, ,Z)$ and 
$\BX_0=\BX_0^+\cup\BX_0^-$, where $\BX_0^+$ and $\BX_0^-$ are defined by formula~\eqref{e:X0+-}.
One has $\BG_m$-equivariant isomorphisms
\begin{equation}
\BA^1\iso\BX_0^+, \; \;  s\mapsto (s,0); \quad\quad\quad \BA^1_-\iso\BX_0^-,  \; \;  s\mapsto (0,s^{-1}).
\end{equation}
They define a morphism
$$\wt{Z}_0=\GMaps (\BX_0\, ,Z)\to\GMaps (\BX_0^+ ,Z)\iso\GMaps (\BA^1 ,Z)=Z^+$$
and a similar morphism $\wt{Z}_0\to Z^-$.

\medskip

By construction, the following diagram commutes: 
\begin{equation}    \label{e:over Z times Z}
\xymatrix{
\wt{Z}_0 \ar[d]^{}\ar[r]^{\wt{p}_0}& Z\times Z\\\
Z^+\underset{Z^0}\times Z^-\ar@{->}[r]^{}&Z^+\times Z^-.\ar[u]_{{p^+}\times {p^-}}
    }
\end{equation}

\sssec{}

We now claim:

\begin{prop}   \label{p:tilde Z_0}
Let $Z$ be a scheme. Then the map \eqref{e:0 fiber} is an isomorphism.
\end{prop}

\begin{proof}

Follows from the fact that for an affine scheme $S$, the diagram
$$
\CD
S\times \on{pt} @>>>  S\times \BX_0^+ \\
@VVV    @VVV    \\
S\times \BX_0^-   @>>>   S\times \BX_0
\endCD
$$
is a push-out diagram \emph{in the category of schemes}.

\end{proof}

\begin{rem}  \label{r:tilde Z_0}
In \cite[Prop. 2.1.11]{Dr} it is shown that the map \eqref{e:0 fiber} is an isomorphism
more generally when $Z$ is an algebraic space.
\end{rem}

\begin{rem} \label{r:tilde Z_0'}
Combining the isomorphism \eqref{e:0 fiber} with the isomorphism $\wt{Z}_1\simeq Z$,
we can interpret $\wt{Z}$ as an $\BA^1$-family\footnote{In general, this ``family" is not flat, see the example from \remref{r:nonflat}.} of spaces interpolating between $Z$ and its ``degeneration" 
$Z^+\underset{Z^0}\times Z^-$. Hence, the title of the subsection.
\end{rem}

\ssec{Basic properties of the interpolation}

\sssec{}

We have:

\begin{prop}   \label{p:closed and open}  \hfill

\smallskip

\noindent{\em(i)} Let $Y\subset Z$ be a $\BG_m$-stable closed subspace. Then the diagram
\[
\xymatrix{
 \wt{Y} \ar[d]_{\wt{p}_Y} \ar[r]^{}& \wt{Z} \ar[d]^{\wt{p}_Z}\\\
\BA^1\times Y\times Y\;\ar@{^{(}->}[r]^{}&\BA^1\times Z\times Z
    }
\]
is Cartesian. In particular, the morphism $\wt{Y}\to\wt{Z}$ is a closed embedding.

\smallskip

\noindent{\em(ii)} Let $Y\subset Z$ be a $\BG_m$-stable open subspace. Then the above diagram identifies
$\wt{Y}$ with an open subspace of the fiber product 
\[
\wt{Z}\underset{\BA^1\times Z\times Z}\times (\BA^1\times Y\times Y)\, .
\]
In particular, the morphism $\wt{Y}\to\wt{Z}$ is an open embedding.

\end{prop}

\begin{proof}

Set $$\oBX:=\BX-\{0\},$$
where $0\in \BX$ is the zero in $\BX=\BA^2$. For $S\to \BA^1$, set
$\oBX_S:=\BX_S\underset{\BX}\times \oBX$.

\medskip

(i) Let $S$ be a scheme over $\BA^1$ and $f:\BX_S\to Z$ a $\BG_m$-equivariant morphism. 
Formula~\eqref{e:two sections} defines two sections of the map $\BX_S\to S$. We have to 
show that if $f$ maps these sections to $Y\subset Z$ then $f(\BX_S)\subset Y$. By 
$\BG_m$-equivariance, we have 
$$f(\oBX_S)\subset Y\,.$$
 
\medskip

Since $\oBX_S$ is schematically dense in $\BX_S$ this implies that
$f(\BX_S)\subset Y$.

\medskip

(ii) Just as before, we have a $\BG_m$-equivariant morphism $f:\BX_S\to Z$ such that 
$f(\oBX)\subset Y$.
The problem is now to show that the set
\[
\{ s\in S\,|\,\BX_s\subset f^{-1}(Y)\}
\]
is open in $S\,$. 

\medskip

The complement of this set equals $\pr_S(\BX_S -f^{-1}(Y))$, where  $\pr_S :\BX_S\to S$ is the projection. 
The set $\pr_S(\BX_S -f^{-1}(Y))$ is closed in $S$ because
$\BX_S -f^{-1}(Y)$ is a closed subset of $\BX_S -\oBX_S$, while 
the morphism $\BX_S -\oBX_S\to S$ is closed (in fact, it is a closed embedding).
\end{proof}

\sssec{}

Next we claim:

\begin{prop}  \label{sss:props tilde p}
Let $Z$ be separated. Then the map 
$$\wt{p}:\wt{Z}\to \BA^1\times Z\times Z$$
is a monomorphism.
\end{prop}

\begin{proof}
As before, set $\oBX:=\BX-\{0\}$, where $0\in \BX$ is the zero in $\BX=\BA^2$.
Given a map $S\to \wt{Z}$, the corresponding $\BG_m$-equivariant map
$$\oBX_S\to Z$$
is completely determined by the composition 
$$S\to \wt{Z} \overset{\wt{p}}\longrightarrow \BA^1\times Z\times Z\,.$$

\medskip

Now use the fact that $\oBX_S$ is schematically dense in $\BX_S$.

\end{proof}

\begin{cor}
If $Z$ is separated then so is $\wt{Z}$. 
\end{cor}

\sssec{The affine case}

We are going to prove:

\begin{prop} \label{p:2new tilde}
Assume that $Z$ is an affine scheme of finite type. Then the morphism $\wt{p}:\wt{Z}\to \BA^1\times Z\times Z$ is a closed embedding. 
In particular, $\wt{Z}$ is an affine scheme of finite type.
\end{prop}

\begin{proof}
If $Z$ is a closed subscheme of an affine scheme $Z'$ and the proposition holds for $Z'$ then it holds for 
$Z$ by  \propref{p:closed and open}(i). So we are reduced to the case that $Z$ is a finite-dimensional 
vector space equipped with a linear $\BG_m$-action.   

\medskip

If the proposition holds for affine schemes $Z_1$ and $Z_2$ then it holds for $Z_1\times Z_2\,$.
So we are reduced to the case that $Z=\BA^1$ and $\lambda\in\BG_m$ acts on $\BA^1$ as multiplication by 
$\lambda^n$, $n\in\BZ$. 

\medskip

In this case it is straightforward to compute $\wt{Z}$ directly. 
In particular, one checks that $\wt{p}$ identifies $\wt{Z}$ with the closed subscheme of 
$\BA^1\times Z\times Z$ defined by the equation $x_2=t^n\cdot x_1$ if $n\ge 0$ and by the equation 
$x_1=t^{-n}\cdot x_2$ if $n\le 0$ (here $t,x_1,x_2$ are the coordinates on $\BA^1\times Z\times Z=\BA^3$).

\end{proof}

\ssec{Representability of the interpolation}  \label{ss:repr of inter}

\sssec{}

We have the following assertion, which is proved in \cite[Thm. 2.2.2 and Prop. 2.2.3]{Dr}:

\begin{thm}   \label{t:tildeZ}
Let $Z$ be an algebraic space (resp., scheme) of finite type equipped with a $\BG_m$-action. Then
$\wt{Z}$ is an algebraic space (resp., scheme) of finite type.
\end{thm}

Below we will give a proof in the case when $Z$ is a scheme and the action of $\BG_m$ on $Z$ is locally linear. 
This case will be sufficient for the applications in this paper. 

\begin{proof}

By assumption, $Z$ can be covered by open affine $\BG_m$-stable subschemes $U_i$. By \propref{p:2new tilde}, each 
$\wt{U}_i$ is an affine scheme of finite type. By \propref{p:closed and open}(ii), for each $i$ the 
canonical morphism $\wt{U}_i\to\wt{Z}$ is an open embedding. It remains to show that $\wt{Z}$ is 
covered by the open subschemes $\wt{U}_i$.

\medskip 

It suffices to check that for each $t\in\BA^1$ the fiber $\wt{Z}_t$ is covered by the open subschemes 
$(\wt{U}_i)_t$. For $t\ne 0$ this is clear from \propref{p:outside 0}. It remains to consider the case $t=0$.

\medskip

By \propref{p:tilde Z_0}, $\wt{Z}_0=Z^+\underset{Z^0}\times Z^-$. So a point of $\wt{Z}_0$ is a pair
$(z^+,z^-)\in Z^+\times Z^-$ such that $q^+(z^+)=q^-(z^-)$. The point $q^+(z^+)=q^-(z^-)$ is contained in some
$U_i\,$. By Lemma~\ref{l:U^+}(i), we have $z^+,z^-\in U_i\,$. So our point $(z^+,z^-)\in\wt{Z}_0$ belongs to 
$ (\wt{U}_i)_0\,$.

\end{proof}

\sssec{The contracting situation}

Let $Z$ be an algebraic space of finite type, and assume that the $\BG_m$-action on $Z$ is contracting,
i.e., the $\BG_m$-action can be extended to an action of the monoid $\BA^1$ (recall that such an extension is unique, see 
\secref{sss:contracting} including \remref{r:contracting}). In this case we claim:

\begin{prop} \label{p:2contracting}  \hfill

\smallskip

\noindent{\em(i)} The morphism $\wt{p}:\wt{Z}\to \BA^1\times Z\times Z$ identifies $\wt{Z}$
with the graph of the $\BA^1$-action on $Z$; in particular, the composition
\begin{equation}  \label{e:first iso}
\wt{Z}\overset{\wt{p}}\longrightarrow\BA^1\times Z\times Z\to\BA^1\times Z\times \on{pt}=\BA^1\times Z
\end{equation}
is an isomorphism.

\smallskip

\noindent{\em(ii)} The inverse of \eqref{e:first iso} is the morphism
\begin{equation}   \label{e:beta}
Z\times \BA^1\to\wt{Z},
\end{equation}
corresponding to the $\BG_m$-equivariant map $Z\times \BX\to Z$, defined by
\[
(z,\tau_1,\tau_2)\mapsto\tau_1\cdot z\, ,\quad\quad (\tau_1,\tau_2)\in\BX\, ,\; z\in Z.
\]
\end{prop}

\begin{proof} 
Let $\alpha :\wt{Z}\to\BA^1\times Z$ denote the composition \eqref{e:first iso} and 
$\beta :\BA^1\times Z\to\wt{Z}$ the morphism~\eqref{e:beta}. It is easy to see that $\alpha\circ\beta=\id$.

\medskip

In order to prove that $\beta\circ \alpha=\id$, it is enough to show that $\alpha$ is a monomorphism. By
\thmref{t:tildeZ}, we are dealing with a morphism between algebraic spaces of finite type, so being a
monomorphism is a fiber-wise condition. Thus, it suffices to show that $\alpha$ induces an isomorphism between fibers 
over any $t\in\BA^1$. 

\medskip

For $t\ne 0$ this follows from \propref{p:outside 0}. If $t=0$ then by \propref{p:tilde Z_0}
(resp., Remark \ref{r:tilde Z_0} in the case of algebraic spaces), the morphism in question is the composition 
$$Z^+\underset{Z^0}\times Z^-\to Z^+\overset{p^+}\longrightarrow Z\,.$$ 
By \propref{p:contracting} (resp., Remark \ref{r:contracting} in the non-separated case), $p^+$ is an isomorphism,
and the projection $q^-:Z^-\to Z^0$ is also an isomorphism by \corref{c:contractive}(i).

\end{proof} 

\sssec{}

From \propref{p:2contracting} we formally obtain the following one:

\begin{prop} \label{p:dilating}  
Let $Z$ be an algebraic space, and assume that the \emph{inverse} of the $\BG_m$-action on $Z$ is contracting. Then:

\smallskip

\noindent{\em(i)} the morphism $\wt{p}:\wt{Z}\to \BA^1\times Z\times Z$ is a monomorphism, which identifies  
$\wt{Z}$ with 
\[
\{(t,z_1,z_2\,)\in\BA^1\times Z\times Z\,|\, z_1=t^{-1}\cdot z_2\,\};
\]
in particular, the composition
\begin{equation} \label{e:second iso}
\wt{Z}\overset{\wt{p}}\longrightarrow\BA^1\times Z\times Z\to\BA^1\times \on{pt}\times Z=\BA^1\times Z
\end{equation}
is an isomorphism.

\smallskip

\noindent{\em(ii)} The inverse of \eqref{e:second iso} is the morphism
\begin{equation}   \label{e:2beta}
Z\times \BA^1\to\wt{Z},
\end{equation}
corresponding to the $\BG_m$-equivariant map $Z\times \BX\to Z$, defined by
\[
(z,\tau_1,\tau_2)\mapsto\tau_2^{-1}\cdot z\, ,\quad\quad (\tau_1,\tau_2)\in\BX\, ,\; z\in Z.
\]
\end{prop}

\ssec{Further properties of the interpolation}

The material in this subsection is included for completeness and will not be used in the sequel.

\medskip

Throughout this subsection, $Z$ will be be an algebraic space of finite type equipped
with a $\BG_m$-action.

\sssec{}

We claim:

\begin{prop}   \label{p:2smoothness}
If $Z$ is smooth then the canonical morphism $\wt{Z}\to\BA^1$ is smooth.
\end{prop}

\begin{proof}
It suffices to check formal smoothness. We proceed just as in the proof of \propref{p:smoothness}.
Let $R$ be a $k$-algebra equipped with a morphism $\Spec(R)\to\BA^1$.  Let $\bar R=R/I$, where 
$I\subset R$ is an ideal with $I^2=0$. Let $\bar f\in\Maps (\BX_{\bar R}\, , Z)^{\BG_m}$. 
We have to show that $\bar f$ can be lifted to 
an element of $\Maps (\BX_R\, , Z)^{\BG_m}$. Since $\BX_R$ is affine and $Z$ is smooth there 
is no obstruction to lifting $\bar f$ to an element of $\Maps (\BX_R, Z)$. The standard arguments 
show that the obstruction to existence of a $\BG_m$-equivariant lift is in $H^1(\BG_m\, ,M)$, where 
$M:=H^0(\BX_{\bar R}\, ,\bar f^*(\Theta_Z))\underset{\bar R}\otimes I$ and $\Theta_Z$ is the tangent bundle of $Z$.
But $H^1$ of $\BG_m$ with coefficients in any $\BG_m$-module is zero.
\end{proof}

\sssec{}

Let $Z$ be affine.  In this case, by \propref{p:2new tilde}, the morphism $\wt{p}$ identifies $\wt{Z}$ with the closed subscheme 
$\wt{p}(\wt{Z})\subset\BA^1\times Z\times Z$. By \propref{p:outside 0}, the intersection of $\wt{p}(\wt{Z})$ with the open subscheme 
$$\BG_m\times Z\times Z\subset \BA^1\times Z\times Z$$
is equal to the graph of the action map $\BG_m\times Z\to Z$.
Hence, $\wt{Z}$ contains the closure of the graph in $\BA^1\times Z\times Z$. 

\begin{rem}  \label{r:nonflat}
In general, this containment is
not an equality. E.g., this happens if  $Z$ is the hypersurface in $\BA^{2n}$ defined by the equation 
$x_1\cdot y_1+\ldots x_n\cdot y_n=0$ and the $\BG_m$-action on $Z$ is defined by 
$\lambda(x_1\, ,\dots,x_n\, ,y_1\, ,\ldots, y_n)=
(\lambda\cdot x_1\, ,\ldots,\lambda\cdot x_n\, ,\lambda^{-1}\cdot y_1\, ,\ldots,\lambda^{-1}\cdot y_n)$.
\end{rem}

However, one has the following:

\begin{prop}
If $Z$ is affine and smooth then 
$$\wt{p}(\wt{Z})=\overline{\Gamma},$$
where $\Gamma\subset\BG_m\times Z\times Z$ is the graph of of the action map $\BG_m\times Z\to Z$
and  $\overline{\Gamma}$ denotes its scheme-theoretical closure in $ \BA^1\times Z\times Z\,$.
\end{prop}

Indeed, this immediately follows from \propref{p:2smoothness}.

\sssec{}

We claim:

\begin{prop}   \label{p:props tilde p'}
The morphism $\wt{p}:\wt{Z}\to \BA^1\times Z\times Z$ is unramified.
\end{prop}

\begin{proof}
The morphism $\wt{p}$ is of finite presentation (because $\wt{Z}$ and $ \BA^1\times Z\times Z$ have finite type over $k$). It remains to check the condition on the geometric fibers of $\wt{p}$.
Over $\BA^1-\{0\}$, it follows from \propref{p:outside 0}.
Over $0\in \BA^1$ it follows from \propref{p:unrami} combined with \propref{p:tilde Z_0} (for schemes) and 
Remark \ref{r:tilde Z_0} (for arbitrary algebraic spaces).
\end{proof}

\sssec{}

Recall that according to \propref{p:2new tilde}, if $Z$ is affine, the map $\wt{p}$ is a closed embedding. 

\medskip

Note, however, that if $Z$ is the projective line $\BP^1$ equipped with the usual $\BG_m$-action then the map
$\wt{p}:\wt{Z}\to \BA^1\times Z\times Z$ is \emph{not a closed} embedding (because, e.g., the scheme
$\wt{Z}_0=Z^+\underset{Z^0}\times Z^-$ is not proper).

\medskip

We have the following assertion:

\medskip

\begin{prop}  \label{p:Pn}
Let $Z$ be a projective space $\BP^n$ equipped with an arbitrary $\BG_m$-action.
Then the morphism $\wt{p}:\wt{Z}\to \BA^1\times Z\times Z$ is a locally closed embedding.
\end{prop}

\begin{proof}
For a suitable coordinate system in $\BP^n$, the $\BG_m$-action is given by
\[
\lambda *(z_0:\ldots :z_n)=(\lambda^{m_0}\cdot z_0: \ldots :\lambda^{m_n}\cdot z_n),\quad \lambda\in\BG_m \, .
\]
Let $U_i\subset Z=\BP^n$ denote the open subset defined by the condition $z_i\ne 0$. It is affine, so by
\propref{p:2new tilde}, the canonical morphism $\wt{U}_i\to\BA^1\times U_i\times U_i$ is a closed embedding.
Thus to finish the proof of the proposition, it suffices to show that 
$\wt{p}^{-1}(\BA^1\times U_i\times U_i)=\wt{U}_i\,$. By \propref{p:outside 0}, 
$\wt{p}^{-1}(\BG_m\times U_i\times U_i)=\BG_m\underset{\BA^1}\times \wt{U}_i\,$. So it remains to prove that the 
morphism $\wt{p}_0:\wt{Z}_0\to Z\times Z$ has the following property:
$(\wt{p}_0)^{-1}(U_i\times U_i)=(\wt{U}_i)_0\,$. Identifying $\wt{Z}_0$ with $Z^+\underset{Z^0}\times Z^-$ and 
using \lemref{l:U^+}(i), we see that it remains to prove the following lemma:

\begin{lem}  \label{l:Pn}
Let $z^+,z^-\in\BP^n$. Suppose that
\[
\lim_{\lambda\to 0}\lambda*z^+=\lim_{\lambda\to\infty }\lambda*z^-=\zeta\, .
\]
If $z^+,z^-\in\ U_i$ then $\zeta\in U_i\,$.
\end{lem}

\end{proof}

\begin{proof}[Proof of \lemref{l:Pn}]

Write $z^+=(z^+_0:\ldots :z^+_n)$, $z^-=(z^-_0:\ldots :z^-_n)$, $\zeta =(\zeta_0:\ldots :\zeta_n)$.
We have $z^{\pm}_i\ne 0$, and the problem is to show that $\zeta_i\ne 0$.

\medskip

Suppose that $\zeta_i= 0$. Choose $j$ so that $\zeta_j\ne 0\,$. Then $z^{\pm}_j\ne 0$ and
\[
\lim_{\lambda\to 0}\lambda^{m_i-m_j}\cdot (z_i/z_j)=\zeta_i/\zeta_j=0, \quad
\lim_{\lambda\to \infty}\lambda^{m_i-m_j}\cdot (z_i/z_j)=\zeta_i/\zeta_j=0\, .
\]
This means that $m_i>m_j$ and $m_i<m_j$ at the same time, which is impossible.
\end{proof}

\sssec{}

As a corollary of \propref{p:Pn}, combined with \propref{p:closed and open}, 
we obtain that if $Z$ admits a $\BG_m$-equivariant locally closed embedding into
a projective space $\BP(V)$, where $\BG_m$ acts linearly on $V$, then the morphism 
$\wt{p}:\wt{Z}\to \BA^1\times Z\times Z$ is a locally closed 
embedding. (Recall, however,  that the map $p^{\pm}:Z^{\pm}\to Z$ is typically not a locally closed embedding, 
see Example \ref{ex:P^1}.) 

\sssec{}

More generally, suppose that the $\BG_m$-action on $Z$ is locally linear. Then the proof of
\thmref{t:tildeZ} shows that in this case the map $\wt{p}$ is, \emph{Zariski locally on the source}, a
locally closed embedding.

\medskip

However, even this is not the case for a general $Z$:

\medskip

Consider the curve obtained from $\BP^1$ by gluing $0$ with $\infty$. 
Equip $Z$ with the $\BG_m$-action induced by the usual action on $\BP^1$. Then $\wt{p}:\wt{Z}\to \BA^1\times Z\times Z$ is 
not a locally closed embedding, locally on the source.  In fact, already $\wt{p}_0:\wt{Z}_0\to Z\times Z$ is not a locally closed embedding 
locally on the source (because the maps $p^{\pm}:Z^{\pm}\to Z$ are not).

\ssec{Some fiber products}  \label{ss:fiber products}

In this subsection we let $Z$ be an algebraic space of finite type, equipped with an action of $\BG_m$.

\sssec{}

In \secref{sss:tilde p} we defined morphisms $\pi_1,\pi_2:\wt{Z}\to Z$.
In \secref{s:Verifying} we will need to consider the fiber product 
\begin{equation}    \label{e:fibered1}
Z^-\underset{Z}\times \wt{Z}\, ,
\end{equation}
formed using $\pi_1:\wt{Z}\to Z$, and the fiber product
\begin{equation}     \label{e:fibered2}
\wt{Z}\underset{Z}\times Z^+ \, ,
\end{equation}
formed using $\pi_2:\wt{Z}\to Z$.

\sssec{}  \label{sss:defining the 2 maps}
Consider the composition
\begin{equation}   \label{e:embedding1}
\BA^1\times Z^+\to\wt{Z^+}=\wt{Z^+}\underset{Z^+}\times Z^+ \to \wt{Z}\underset{Z}\times Z^+,
\end{equation}
where the first arrow is the morphism \eqref{e:beta} for the space $Z^+$ and the second arrow is induced by 
the morphism $p^+:Z^+\to Z$. Consider also the similar composition
\begin{equation}   \label{e:embedding2}
\BA^1\times Z^-\to\wt{Z^-}=Z^-\underset{Z^-}\times\wt{Z^-}  \to Z^-\underset{Z}\times\wt{Z},
\end{equation}
where the first arrow is the morphism \eqref{e:2beta} for the space $Z^-$. 
In \secref{s:Verifying} we will need the following result.

\begin{prop}   \label{p:2open embeddings}
The compositions \eqref{e:embedding1} and \eqref{e:embedding2} are open embeddings.
\end{prop}

Note that unlike the situation of \propref{p:Cartesian}, these embeddings are usually not closed.

\begin{rem}
By Propositions~\ref{p:2contracting} and \ref{p:dilating}, the maps $\BA^1\times Z^+\to\wt{Z^+}$ and 
$\BA^1\times Z^-\to\wt{Z^-}$ are isomorphisms, so \propref{p:2open embeddings} means that the morphisms
\[
\wt{Z^+}\to\wt{Z}\underset{Z}\times Z^+ ,\quad \wt{Z^-}\to Z^-\underset{Z}\times\wt{Z} 
\]
are open embeddings.
\end{rem}

\begin{rem}
In the course of the proof of \propref{p:2open embeddings} we will see that if $Z$ is affine,
then the maps \eqref{e:embedding1} and \eqref{e:embedding2} are isomorphisms.
\end{rem}

\sssec{}

We will prove \propref{p:2open embeddings} assuming that the action of $\BG_m$ on $Z$ is 
locally linear.  The general case is considered in \cite[Prop. 3.1.3]{Dr}.

\medskip

We will show that \eqref{e:embedding1} is an open embedding. The case of \eqref{e:embedding2} 
is similar.

\begin{proof}

First, \propref{p:closed and open}(ii) and \lemref{l:U^+}(i) allow to 
reduce the assertion to the case when $Z$ is affine.  In the
affine case we will show that the map \eqref{e:embedding1} is an isomorphism.

\medskip

Next, it follows from  \propref{p:closed and open}(i) and \lemref{l:U^+}(ii)
that if $Z\to Z'$ is a closed embedding
and \eqref{e:embedding1} is an isomorphism for $Z'$, then it is also an isomorphism for $Z$.
This reduces the assertion to the case when $Z$ is a vector space equipped with a linear action
of $\BG_m$.

\medskip

Third, it is easy to see that if $Z=Z_1\times Z_2$, and \eqref{e:embedding1} is an isomorphism for $Z_1$
and $Z_2$, then it is an isomorphism for $Z$. This reduces the assertion further to the case when either
the action of $\BG_m$ on $Z$ or its inverse is contracting.

\medskip

Suppose that the action is contracting. In this case $Z^+\simeq Z$ by \propref{p:contracting},
and under this identification the map 
$$\wt{Z^+}\underset{Z^+}\times Z^+ \to \wt{Z}\underset{Z}\times Z^+,$$
appearing in \eqref{e:embedding1}, is the identity map. 

\medskip

Suppose that the inverse of the given $\BG_m$-action on $Z$ is contracting. By \corref{c:contractive}(ii),
we can identify $Z^+\simeq Z^0$, and by \propref{p:dilating}, $\wt{Z}\simeq \BA^1\times Z$. Under
these identifications, the map \eqref{e:embedding1} is the identity map
$$\BA^1\times Z^0\to (\BA^1\times Z)\underset{Z}\times Z^0\simeq \BA^1\times Z^0\,.$$

\end{proof}

\section{Braden's theorem} \label{s:Braden1}

From now on we will assume that the ground field $k$ has characteristic 0 (because we will be working with D-modules).

\medskip

The goal of this section is to state Braden's theorem (\thmref{t:braden original}) in the context of D-modules,
and reduce it to another statement (\thmref{t:Braden adj}) that says that certain two functors are adjoint.

\medskip

Braden's theorem applies to any algebraic space $Z$ of finite type over $k$, equipped with an action
of $\BG_m$. The reader may prefer to restrict his attention to the case of $Z$ being a scheme or even a separated
scheme. 

\medskip

Furthermore, because of Remark \ref{r:q-proj loc lin}, for most applications, it is sufficient to consider the case 
when the  $\BG_m$-action on $Z$ is locally linear, which would make the present paper self-contained, as the main technical results in Sects. 
\ref{s:actions}-\ref{s:deg} were proved only in this case.  

\ssec{Statement of Braden's theorem: the original formulation}  \label{ss:original statement}

\sssec{}  \label{sss:mon}

Let $G$ be an algebraic group. If $Z$ is an algebraic space of finite type equipped with a $G$-action, then
$$\Dmod(Z)^{G\on{-mon}}\subset \Dmod(Z)$$ stands for the
full subcategory generated by the essential image of the pullback functor 
$\Dmod(Z/G)\to \Dmod(Z)$, where $Z/G$ denotes the quotient \emph{stack}. Here one can use either the $!$- or the $\bullet$-pullback: this
makes no difference as the morphism $Z\to Z/\BG_m$ is smooth, and hence the two pullback functors differ by the cohomological shift by $2\cdot\dim(G)$.

\medskip

Note that if the $G$-action is trivial then 
$\Dmod(Z)^{G\on{-mon}}=\Dmod(Z)$ (because the morphism $Z\to Z/G$ admits a section).

\sssec{}

From now on let $Z$ be an algebraic space of finite type equipped with a $\BG_m$-action. Consider the commutative diagram
\begin{equation} \label{e:square with arrow}
\xy
(0,0)*+{Z^+\underset{Z}\times Z^-}="X";
(30,0)*+{Z^+}="Y";
(0,-30)*+{Z^-}="Z";
(30,-30)*+{Z.}="W";
(-20,20)*+{Z^0}="U";
{\ar@{->}_{p^-} "Z";"W"};
{\ar@{->}^{p^+} "Y";"W"};
{\ar@{->}_{'p^-} "X";"Y"};
{\ar@{->}^{'p^+} "X";"Z"};
{\ar@{->}_{j} "U";"X"};
{\ar@{->}_{i^-} "U";"Z"};
{\ar@{->}^{i^+} "U";"Y"};
\endxy
\end{equation}
(The definitions of $Z^0$, $Z^{\pm}$, $i^{\pm}$, and $p^{\pm}$ were given in Sects.~\ref{ss:fixed_points}, 
\ref{ss:attr}, and \ref{ss:repeller}.) 

\medskip

Recall that by \propref{p:Cartesian}, the morphism 
$j:Z^0\to Z^+\underset{Z}\times Z^-$ is an open embedding (and also a closed one).

\medskip

We consider the categories
$$\Dmod(Z)^{\BG_m\on{-mon}},\,\, \Dmod(Z^+)^{\BG_m\on{-mon}},\,\, \Dmod(Z^-)^{\BG_m\on{-mon}}$$ and   
$$\Dmod(Z^0)^{\BG_m\on{-mon}}=\Dmod(Z^0).$$

Consider the functors
$$(p^+)^!:\Dmod(Z)^{\BG_m\on{-mon}}\to \Dmod(Z^+)^{\BG_m\on{-mon}} \text{ and } 
(i^-)^!:\Dmod(Z^-)^{\BG_m\on{-mon}}\to \Dmod(Z^0).$$

The formalism of pro-categories (see Appendix \ref{s:pro}) also provides the functors 

$$(p^-)^\bullet: \Dmod(Z)^{\BG_m\on{-mon}}\to \on{Pro}(\Dmod(Z^-)^{\BG_m\on{-mon}})$$  and 
$$(i^+)^\bullet: \Dmod(Z^+)^{\BG_m\on{-mon}}\to \on{Pro}(\Dmod(Z^0)),$$
left adjoint in the sense of Sect. A.3 to
$$(p^-)_\bullet:\Dmod(Z^-)^{\BG_m\on{-mon}}\to \Dmod(Z)^{\BG_m\on{-mon}}$$ and 
$$(i^+)_\bullet: \Dmod(Z^0)\to  \Dmod(Z^+)^{\BG_m\on{-mon}},$$
respectively. 

\sssec{}

Consider the composed functors
$$(i^+)^\bullet\circ (p^+)^! \text{ and } (i^-)^!\circ (p^-)^\bullet,\quad 
\Dmod(Z)^{\BG_m\on{-mon}}\to \on{Pro}(\Dmod(Z^0)).$$

\medskip

They are called the functors of \emph{hyperbolic restriction}.  

\sssec{}

We claim that there is a canonical natural transformation
\begin{equation} \label{e:Braden trans}
(i^+)^\bullet\circ (p^+)^! \to (i^-)^!\circ (p^-)^\bullet.
\end{equation}
Namely, the natural transformation \eqref{e:Braden trans} is obtained via the $((i^+)^\bullet,(i^+)_\bullet)$-adjunction 
from the natural transformation
\begin{equation}  \label{e:the natural transformation}
(p^+)^! \to (i^+)_\bullet\circ (i^-)^!\circ (p^-)^\bullet,
\end{equation}
defined in terms of diagram \eqref{e:square with arrow} as follows. 

\medskip

Note that since $j:Z^0\to Z^+\underset{Z}\times Z^-$ is an \emph{open embedding} (see \propref{p:Cartesian}), 
the functor $j^!$ is left adjoint to 
$j_\bullet\,$. Now define the morphism \eqref{e:the natural transformation} to be the composition
\begin{multline*}
(p^+)^!\to (p^+)^!\circ (p^-)_\bullet \circ (p^-)^\bullet \simeq ({}'p^-)_\bullet\circ ({}'p^+)^!\circ (p^-)^\bullet\to \\
\to ({}'p^-)_\bullet\circ j_\bullet\circ j^! \circ ({}'p^+)^! \circ (p^-)^\bullet
\simeq  (i^+)_\bullet\circ (i^-)^!\circ (p^-)^\bullet,
\end{multline*}
where $(p^+)^!\circ (p^-)_\bullet\simeq ({}'p^-)_\bullet\circ ({}'p^+)^!$
is the base change isomorphism and the map 
$$\on{Id}\to j_\bullet\circ j^!$$
comes from the $(j^!,j_\bullet)$-adjunction.

\sssec{}

We are now ready to state Braden's theorem:

\begin{thm} \label{t:braden original}  
The functors 
$$(i^+)^\bullet\circ (p^+)^! \text{ and } (i^-)^!\circ (p^-)^\bullet, \quad \Dmod(Z)^{\BG_m\on{-mon}}\to \on{Pro}(\Dmod(Z^0))$$
take values in $\Dmod(Z^0)\subset \on{Pro}(\Dmod(Z^0))$
and the map
\eqref{e:Braden trans}
is an isomophism.
\end{thm}

\begin{rem}
As we will see in \secref{sss:contr}, the fact that the functor $(i^+)^\bullet\circ (p^+)^!$ takes values in 
$\Dmod(Z^0)\subset \on{Pro}(\Dmod(Z^0))$ is easy to prove. 
The fact that the functor $(i^-)^!\circ (p^-)^\bullet$
takes values in $\Dmod(Z^0)$ will follow  \emph{a posteriori} from the isomorphism with 
$(i^+)^\bullet\circ (p^+)^!$. 
\end{rem}

\ssec{Contraction principle} 

\sssec{}

Assume for a moment that the $\BG_m$-action on $Z$ extends\footnote{By Remark \ref{r:contracting}, such extension is unique if it exists.} 
to an action of the monoid $\BA^1$. (Informally, this means 
that the $\BG_m$-action on $Z$ contracts it onto the fixed point locus $Z^0$.)

\begin{prop}  \label{p:simple Braden}  
In the above situation we have the following:
\smallskip

\noindent{\em(a)} The left adjoint $i^\bullet:\Dmod(Z)\to \on{Pro}(\Dmod(Z^0))$ of $i_\bullet$ sends
$\Dmod(Z)^{\BG_m\on{-mon}}$ to $\Dmod(Z^0)$, and we have a canonical isomorphism
$$i^\bullet|_{\Dmod(Z)^{\BG_m\on{-mon}}}\simeq q_\bullet|_{\Dmod(Z)^{\BG_m\on{-mon}}} \; .$$
More precisely, for each $\CF\in \Dmod(Z)^{\BG_m\on{-mon}}$
the natural map
$$q_\bullet(\CF)\to q_\bullet\circ i_\bullet\circ i^\bullet(\CF)=(q\circ i)_\bullet\circ i^\bullet(\CF)=i^\bullet(\CF)$$
is an isomorphism.

\smallskip

\noindent{\em(b)} The left adjoint $q_!:\Dmod(Z)\to \on{Pro}(\Dmod(Z^0))$ of $q^!$ sends
$\Dmod(Z)^{\BG_m\on{-mon}}$ to $\Dmod(Z^0)$, and we have a canonical isomorphism
$$q_!|_{\Dmod(Z)^{\BG_m\on{-mon}}}\simeq i^!|_{\Dmod(Z)^{\BG_m\on{-mon}}} \; .$$
More precisely, for each $\CF\in \Dmod(Z)^{\BG_m\on{-mon}}$
the natural map
$$i^!(\CF)\to i^!\circ q^!\circ q_!(\CF) =(q\circ i)^!\circ q_!(\CF) =q_! (\CF)$$
is an isomorphism.
\end{prop}  

For the proof see \cite[Theorem C.5.3]{DrGa2}. 

\sssec{}

Note that we can reformulate point (a) of \propref{p:simple Braden} above as the statement that 
the (iso)morphism
$$q_\bullet\circ i_\bullet\to \on{Id}_{\Dmod(Z^0)}$$
defines the co-unit of an adjunction between
$$q_\bullet:\Dmod(Z)^{\BG_m\on{-mon}}\rightleftarrows \Dmod(Z^0):i_\bullet.$$

\medskip

Similarly, point (b) of \propref{p:simple Braden} can be reformulated as the statement that 
the (iso)morphism
$$i^!\circ q^! \to \on{Id}_{\Dmod(Z^0)}$$
defines the co-unit of an adjunction between 
$$i^!:\Dmod(Z)^{\BG_m\on{-mon}} \rightleftarrows \Dmod(Z^0):q^!.$$

\ssec{Reformulation of Braden's theorem} \label{ss:Reformulation of Braden}

\sssec{}  \label{sss:contr}

We return to the set-up of \thmref{t:braden original}.  By \propref{p:simple Braden}, we obtain 
canonical isomorphisms
$$(i^+)^\bullet\simeq (q^+)_\bullet \text{ and } (i^-)^!\simeq (q^-)_!.$$

In particular, we obtain that the functor 
$$(i^+)^\bullet\circ (p^+)^!\simeq (q^+)_\bullet \circ (p^+)^!$$
sends $\Dmod(Z)^{\BG_m\on{-mon}}$ to $\Dmod(Z^0)$.

\medskip

In addition, we see that the functor 
$$(i^-)^!\circ (p^-)^\bullet\simeq (q^-)_!\circ (p^-)^\bullet$$
is the left adjoint functor to $(p^-)_\bullet\circ (q^-)^!$. 

\sssec{}   \label{sss:defining co-unit}

Now define a natural transformation
\begin{equation} \label{e:Braden co-unit}
\left((q^+)_\bullet \circ (p^+)^!\right)\circ \left((p^-)_\bullet\circ (q^-)^!\right)\to \on{Id}_{\Dmod(Z^0)}
\end{equation}
to be the composition
\begin{multline*}
(q^+)_\bullet \circ (p^+)^! \circ (p^-)_\bullet\circ (q^-)^!\simeq 
(q^+)_\bullet \circ  ({}'p^-)_\bullet\circ ({}'p^+)^! \circ (q^-)^! \to \\
\to (q^+)_\bullet \circ  ({}'p^-)_\bullet\circ j_\bullet\circ j^! \circ ({}'p^+)^! \circ (q^-)^! 
\simeq (q^+)_\bullet \circ (i^+)_\bullet \circ (i^-)^!\circ  (q^-)^! \simeq \\
\simeq (q^+\circ i^+)_\bullet\circ (q^-\circ i^-)^!=
\on{Id}_{\Dmod(Z^0)}.
\end{multline*}

The above natural transformation corresponds to the diagram

$$
\xy
(-20,0)*+{Z}="X";
(20,0)*+{Z^0.}="Y";
(0,20)*+{Z^+}="Z";
(-40,20)*+{Z^-}="W";
(-60,0)*+{Z^0}="U";
(-20,40)*+{Z^-\underset{Z}\times Z^+}="V";
(-20,85)*+{Z^0}="T";
{\ar@{->}^{p^+} "Z";"X"};
{\ar@{->}_{q^+} "Z";"Y"};
{\ar@{->}_{p^-} "W";"X"};
{\ar@{->}^{q^-} "W";"U"}; 
{\ar@{->}_{'p^-} "V";"Z"};
{\ar@{->}^{'p^+} "V";"W"};
{\ar@{->}_{j} "T";"V"};
{\ar@{->}^{i^-} "T";"W"};
{\ar@{->}_{i^+} "T";"Z"};
{\ar@{->}^{\on{id}} "T";"Y"};
{\ar@{->}_{\on{id}} "T";"U"};
\endxy
$$

\sssec{}

The natural transformation \eqref{e:Braden co-unit} gives rise to (and is determined by) a natural transformation 
\begin{equation} \label{e:reform Braden trans}
(q^+)_\bullet \circ (p^+)^!\to \left((p^-)_\bullet\circ (q^-)^!\right)^L\simeq (q^-)_!\circ (p^-)^\bullet.
\end{equation}
Here $\left((p^-)_\bullet\circ (q^-)^!\right)^L$ denotes the left adjoint of 
$(p^-)_\bullet\circ (q^-)^!$ in the sense of Sect. A.3.

\medskip

It follows by diagram chase that the following diagram of natural transfomations commutes:
$$
\CD
(q^+)_\bullet \circ (p^+)^!   @>{\text{\eqref{e:reform Braden trans}}}>>  (q^-)_!\circ (p^-)^\bullet   \\
@A{\sim}AA   @A{\sim}AA  \\
(i^+)^\bullet \circ (p^+)^!   @>{\text{\eqref{e:Braden trans}}}>>  (i^-)^!\circ (p^-)^\bullet 
\endCD
$$

\medskip

Hence, the assertion of \thmref{t:braden original} follows from the next one:

\begin{thm} \label{t:Braden adj}
The natural transformation \eqref{e:Braden co-unit} is the co-unit
of an adjunction for the functors
$$(q^+)_\bullet \circ (p^+)^!:\Dmod(Z)^{\BG_m\on{-mon}}\rightleftarrows \Dmod(Z^0):(p^-)_\bullet\circ (q^-)^!$$
\end{thm}

\ssec{The equivariant version}   \label{ss:equivariant version}

\sssec{}   \label{sss:Consider now}

Consider now the stacks 
$$\CZ:=Z/\BG_m\, ,\,\, \CZ^0:=Z^0/\BG_m\, ,\,\, \CZ^{\pm}:=Z^{\pm}/\BG_m$$  
and the morphisms
$$ \sfp^{\pm}:\CZ^{\pm}\to \CZ\, ,\,\,\sfq^{\pm}:\CZ^{\pm}\to\CZ^0$$
induced by the morphisms
$$p^{\pm}:Z^{\pm}\to Z ,\,\,q^{\pm}:Z^{\pm}\to Z^0$$
from Sects.~\ref{sss:structures} and \ref{ss:repeller}.

\sssec{}

The construction of the natural transformation \eqref{e:Braden co-unit} can be rendered verbatim to produce a natural transformation
\begin{equation} \label{e:Braden co-unit equiv}
\left((\sfq^+)_\bullet \circ (\sfp^+)^!\right)\circ \left((\sfp^-)_\bullet\circ (\sfq^-)^!\right)\to \on{Id}_{\Dmod(\CZ^0)}.
\end{equation}

\medskip

We will prove the following version of \thmref{t:Braden adj}:

\begin{thm} \label{t:Braden adj equiv}
The natural transformation \eqref{e:Braden co-unit equiv} is the co-unit
of an adjunction for the functors
$$(\sfq^+)_\bullet \circ (\sfp^+)^!:\Dmod(\CZ)\rightleftarrows \Dmod(\CZ^0):(\sfp^-)_\bullet\circ (\sfq^-)^!$$
\end{thm}

Let us prove that \thmref{t:Braden adj equiv} implies \thmref{t:Braden adj}. 

\begin{proof}

We need to show that for $\CM\in \Dmod(Z)^{\BG_m\on{-mon}}$ and $\CN\in \Dmod(Z^0)^{\BG_m\on{-mon}}$, the map
$$\Hom_{\Dmod(Z)^{\BG_m\on{-mon}}}\left(\CM,(p^-)_\bullet\circ (q^-)^!(\CN)\right)\to
\Hom_{\Dmod(Z^0)^{\BG_m\on{-mon}}}\left((q^+)_\bullet \circ (p^+)^!(\CM),\CN\right),$$
induced by \eqref{e:Braden co-unit}, is an isomorphism.

\medskip

By the definition of $\Dmod(Z)^{\BG_m\on{-mon}}$, we can assume that $\CM$ is the $\bullet$-pullback of some
$\CM'\in \Dmod(\CZ)$. Let $\CN'$ denote the $\bullet$-direct image of $\CN$ under the canonical map $Z^0\to \CZ^0$.

\medskip

Since all the maps $Z\to \CZ$, $Z^0\to \CZ^0$ and $Z^{\pm}\to \CZ^{\pm}$
are smooth, we have the following commutative diagram (with the vertical arrows being isomorphisms by adjunction): 
$$
\CD
\Hom\left(\CM,(p^-)_\bullet\circ (q^-)^!(\CN)\right)  @>>> 
\Hom\left((q^+)_\bullet \circ (p^+)^!(\CM),\CN\right)  \\
@V{\sim}VV   @VV{\sim}V  \\
\Hom\left(\CM',(\sfp^-)_\bullet\circ (\sfq^-)^!(\CN')\right)   @>>>
\Hom\left((\sfq^+)_\bullet \circ (\sfp^+)^!(\CM'),\CN'\right).
\endCD
$$

Hence, if the bottom horizontal arrow is an isomorphism, then so is the top one. 

\end{proof}

\section{Construction of the unit}  \label{s:unit}

In this section we will perform the main step in the proof of \thmref{t:Braden adj equiv}; namely, we will
construct the \emph{unit} for the adjunction between the functors 
$(\sfq^+)_\bullet \circ (\sfp^+)^!$ and $(\sfp^-)_\bullet\circ (\sfq^-)^!$.

\ssec{The specialization map}   \label{ss:specialization}

In this subsection we describe the general set-up for the specialization map.
The concrete situation in which this set-up will 
be applied is described in Sects. \ref{sss:concrete situation}-\ref{sss:concrete} below.

\sssec{}

Let $\CY$ be an algebraic \emph{stack} \footnote{We use the conventions from \cite[Sect. 1.1]{DrGa1}
for algebraic stacks. We refer the reader to \cite[Sect. 6]{DrGa1} for a review of the DG category
of D-modules on algebraic stacks of finite type.} of finite type. Consider the stack $\BA^1\times \CY$, and let 
$\iota_1$ and $\iota_0$ be the maps $\CY\to \BA^1 \times \CY$ corresponding to 
the points $1$ and $0$ of $\BA^1$, respectively. Let $\pi$ denote the projection $\BA^1\times \CY\to \CY$. 

\medskip

Let $\CK$ be an object of $\Dmod(\BA^1\times \CY)^{\BG_m\on{-mon}}$, where 
$$\Dmod(\BA^1\times \CY)^{\BG_m\on{-mon}}\subset \Dmod(\BA^1\times \CY)$$
is the full subcategory generated by the essential image of the pullback functor
$$\Dmod\left((\BA^1/\BG_m)\times \CY\right)\to \Dmod(\BA^1\times \CY).$$

\medskip

Set
$$\CK_1:=\iota_1^!(\CK), \quad\quad \CK_0:=\iota_0^!(\CK).$$

\medskip

We are going to construct a canonical map
\begin{equation} \label{e:specialization}
\on{Sp}_\CK:\CK_1\to \CK_0\, ,
\end{equation} 
which will depend functorially on $\CK$. We will call it the \emph{specialization map}.

\begin{rem}
The map \eqref{e:specialization} is a simplified version of the specialization map
that goes from the nearby cycles functor to the !-fiber. 
\end{rem} 

\sssec{}

First, note that \propref{p:simple Braden}(b) and the definition of the category $\Dmod(-)$ for an algebraic stack 
\footnote{According to \cite[Sect. 6.1.1]{DrGa1}, an object of $\Dmod(\CY )$ is a ``compatible collection" of objects of $\Dmod(S )$ 
for all schemes $S$ of finite type mapping to $\CY$.}  imply that the functor $\pi_!$, left adjoint to 
$\pi^!:\Dmod(\CY)\to \Dmod(\BA^1\times \CY)$, is defined on the subcategory
$\Dmod(\BA^1\times \CY)^{\BG_m\on{-mon}}$, and the natural transformation
$$\iota_0^!\to \iota_0^!\circ \pi^!\circ \pi_!\simeq \pi_!$$
is an isomorphism. 

\medskip

Now, we construct the natural transformation \eqref{e:specialization} as
$$\iota_1^!(\CK)\simeq \pi_!\circ (\iota_1)_!\circ \iota_1^!(\CK) \to \pi_!(\CK)\simeq \iota_0^!(\CK),$$
where the morphism  $\pi_!\circ (\iota_1)_!\circ \iota_1^!(\CK)\to \pi_!(\CK)$ comes from the
$((\iota_1)_!,\iota_1^!)$-adjunction. Note that the functor $(\iota_1)_!$ is well-defined 
because $\iota_1$ is a closed embedding.

\sssec{}  \label{sss:specialization for constant}

It is easy to see that if $\CK=\omega_{\BA^1}\times \CK_\CY$ for some $\CY\in \Dmod(\CY)$,
then the map \eqref{e:specialization} is the identity endomorphism of
$$\iota_1^!(\CK)\simeq \CK_\CY\simeq \iota_0^!(\CK).$$

\sssec{}  \label{sss:functoriality of specialization}

It is also easy to see from the construction that the natural transformation \eqref{e:specialization}
is functorial with respect to maps between algebraic stacks in the following sense.

\medskip

Let $f:\CY'\to \CY$ be a map. Then for $\CK':=(\id_{\BA^1}\times f)^!(\CK)$ the diagram
$$
\CD
\CK'_1  @>{\on{Sp}_{\CK'}}>>  \CK'_0  \\
@A{\sim}AA   @AA{\sim}A  \\ 
f^!(\CK_1)  @>{f^!(\on{Sp}_\CK)}>>   f^!(\CK_0)  \\
\endCD
$$
commutes. 

\medskip

Let now $f$ be representable and quasi-compact. Then for $\CK'\in \Dmod(\BA^1\times \CY')^{\BG_m\on{-mon}}$
and 
$$\CK:=(\id_{\BA^1}\times f)_\bullet(\CK'),$$ the diagram
$$
\CD
\CK_1  @>{\on{Sp}_\CK}>>  \CK_0  \\
@V{\sim}VV   @VV{\sim}V  \\ 
f_\bullet(\CK'_1)  @>{f_\bullet(\on{Sp}_{\CK'})}>>   f_\bullet(\CK'_0)  \\
\endCD
$$
also commutes. 

\ssec{Digression: functors given by kernels} \label{ss:kernels}

\sssec{}
According to  \cite[Definition 1.1.8]{DrGa1}, an algebraic stack of finite type over $k$ is said to be QCA if the automorphism 
groups of its geometric points are affine.

\medskip

If $f:\CY\to \CY'$ is a morphism between QCA stacks then one has a canonically defined functor
$$f_\blacktriangle:\Dmod(\CY)\to \Dmod(\CY')$$
defined  in \cite[Sect. 9.3]{DrGa1}.

\begin{rem}
The functor $f_\blacktriangle$ is a ``renormalized version" of the usual functor $f_\bullet$ of de Rham
direct image (see \cite[Sect. 7.4]{DrGa1}). The problem with the functor $f_\bullet$ is that it is very
poorly behaved unless the morphism $f$ is representable \footnote{Or, more generally, \emph{safe}
in the sense of \cite[Definition 10.2.2]{DrGa1}.}. For example, it fails to satisfy the projection
formula and \emph{is not compatible with compositions},  see \cite[Sect. 7.5]{DrGa1} for more
details. The functor $f_\blacktriangle$ cures all these drawbacks, and it equals the usual functor 
$f_\bullet$ if $f$ is representable.
\end{rem}

\sssec{}  \label{sss:functors and kernels}

Let $\CY_1$ and $\CY_2$ be QCA algebraic stacks. For an object $\CQ\in \Dmod(\CY_1\times \CY_2)$, consider the functor
$$\sF_\CQ:\Dmod(\CY_1)\to \Dmod(\CY_2),\quad \CM\mapsto (\on{pr}_2)_\blacktriangle(\on{pr}_1^!(\CM)\sotimes \CQ),$$
where $\on{pr}_i:\CY_1\times \CY_2\to \CY_i$ are the two projections, and $\sotimes$ is the usual tensor product
on the category of D-modules. 

\medskip

We will refer to $\CQ$ as the \emph{kernel} of the functor $\sF_\CQ$. 

\medskip

In fact, it follows from \cite[Corollary 8.3.4]{DrGa1}
that the assignment $\CQ\rightsquigarrow \sF_\CQ$ defines an equivalence between the category $\Dmod(\CY_1\times \CY_2)$
and the DG category of \emph{continuous} \footnote{Recall that a functor between cocomplete DG categories is said to be 
continuous if it commutes with arbitrary direct sums.} functors $\Dmod(\CY_1)\to \Dmod(\CY_2)$. 

\medskip

For example, if $\CY_1=\CY_2=\CY$, then for
$$\CQ:=(\Delta_\CY)_\blacktriangle(\omega_\CY)\in \Dmod(\CY\times \CY)$$
the corresponding functor $\sF_\CQ$ is the identity functor on $\Dmod(\CY)$. 
Here $\omega_{\CY}\in \Dmod(\CY)$ denotes the dualizing complex on a stack $\CY$. 

\sssec{}   \label{sss:corr}

More generally, let
\begin{equation} \label{e:corr}
\xy
(-15,0)*+{\CY_1}="X";
(15,0)*+{\CY_2}="Y";
(0,15)*+{\CY_0}="Z";
{\ar@{->}_{f_1} "Z";"X"};
{\ar@{->}^{f_2} "Z";"Y"};
\endxy
\end{equation}
be a diagram of QCA algebraic stacks. Set
$$\CQ:=(f_1\times f_2)_\blacktriangle(\omega_{\CY_0})\in \Dmod(\CY_1\times \CY_2).$$

\medskip

Then, by the projection formula, the functor $\sF_\CQ$ identifies with $(f_2)_\blacktriangle\circ (f_1)^!$.

\sssec{}   \label{sss:two routes}

The reader who is reluctant to use the (potentially unfamiliar) functor $f_\blacktriangle$ can proceed along either
of the following two routes: 

\medskip

\noindent(i) The usual functor of direct image $f_\bullet$ is well-behaved when restricted to the subcategory
$\Dmod(\CY)^+$ of bounded below (=eventually coconnective) objects. It is easy to see that working with
this subcategory would be sufficient for the proof of \thmref{t:Braden adj equiv}.  \footnote{Note, however,
that if one redefines the assignment $\CQ\rightsquigarrow \sF_\CQ$ using $(\on{pr}_2)_\bullet$ instead of
$(\on{pr}_2)_\blacktriangle$ then one obtains a \emph{different} functor, even when evaluated on 
$\Dmod(\CY_1)^+$.}

\medskip

This strategy can be used in order to adapt the proof of \thmref{t:Braden adj equiv} to the context of $\ell$-adic
sheaves. 

\medskip

\noindent(ii) One can use the following assertion.
\begin{lem}
Suppose that the morphism $f_2:\CY_0\to\CY_2$ is representable. Then 

\smallskip

\noindent{\em(i)} 
The kernel $\CQ:=(f_1\times f_2)_\blacktriangle(\omega_{\CY_0})$ is canonically isomorphic to 
$(f_1\times f_2)_\bullet(\omega_{\CY_0})$; 

\smallskip

\noindent{\em(ii)} The functor
$$\sF_\CQ:\Dmod(\CY_1)\to \Dmod(\CY_2),\quad 
\CM\mapsto (\on{pr}_2)_\blacktriangle(\on{pr}_1^!(\CM)\sotimes \CQ)\simeq (f_2)_\blacktriangle\circ (f_1)^! (\CM )$$
is canonically isomorphic  to the functor
$$\CM\mapsto (\on{pr}_2)_\bullet(\on{pr}_1^!(\CM)\sotimes \CQ).$$
\end{lem}


\begin{proof}
Since $f_2:\CY_0\to\CY_2$ is representable, so is the morphism 
$f_1\times f_2:\CY_0\to\CY_1\times\CY_2\,$. This implies (i). 

\medskip

We have canonical isomorphisms
$$\on{pr}_1^!(\CM)\sotimes \CQ\simeq (f_1\times f_2)_\blacktriangle (f_1^!(\CM ))\simeq (f_1\times f_2)_\bullet (f_1^!(\CM ))$$
(the first one holds by projection formula and the second because $f_1\times f_2$ is representable). So
$$(\on{pr}_2)_\bullet(\on{pr}_1^!(\CM)\sotimes \CQ)\simeq
((\on{pr}_2)_\bullet\circ (f_1\times f_2)_\bullet) (f_1^!(\CM )).$$
One also has
$$\sF_\CQ (\CM )\simeq  (f_2)_\blacktriangle\circ (f_1)^! (\CM )\simeq (f_2)_\bullet (f_1^!(\CM ))=
((\on{pr}_2)\circ (f_1\times f_2))_\bullet (f_1^!(\CM )).$$

Finally, the fact that $f_1\times f_2$ is representable (see \cite[Proposition 7.5.7]{DrGa1}
\footnote{For any composable morphisms $g,g'$ between stacks one has a morphism $g_\bullet\circ g'_\bullet\to (g\circ g')_\bullet\,$, 
which is \emph{not necessarily an isomorphism}. However, it is an isomorphism if $g'$ is representable. In 
\cite[Proposition 7.5.7]{DrGa1} this is proved if $g'$ is schematic, but the same proof applies if $g'$ is only representable.}) implies that 
$$((\on{pr}_2)\circ (f_1\times f_2))_\bullet\simeq (\on{pr}_2)_\bullet\circ (f_1\times f_2)_\bullet\,.$$
\end{proof}

\ssec{The unit of adjunction: plan of the construction}  \sssec{} \label{sss:reformulating}
In \secref{sss:Consider now} we introduced the stacks
$$\CZ:=Z/\BG_m\, ,\,\, \CZ^0:=Z^0/\BG_m\, ,\,\, \CZ^{\pm}:=Z^{\pm}/\BG_m$$
and the morphisms
$\sfp^{\pm}:\CZ^{\pm}\to \CZ\,$, $\sfq^{\pm}:\CZ^{\pm}\to\CZ^0.$ 
Now consider the diagram
\begin{equation} \label{e:unit diagram}
\xy
(-20,0)*+{\CZ^0}="X";
(20,0)*+{\CZ.}="Y";
(0,20)*+{\CZ^-}="Z";
(-40,20)*+{\CZ^+}="W";
(-60,0)*+{\CZ}="U";
(-20,40)*+{\CZ^+\underset{\CZ^0}\times \CZ^-}="V";
{\ar@{->}_{\sfq^-} "Z";"X"};
{\ar@{->}^{\sfp^-} "Z";"Y"};
{\ar@{->}^{\sfq^+} "W";"X"};
{\ar@{->}_{\sfp^+} "W";"U"}; 
{\ar@{->}^{'\sfq^+} "V";"Z"};
{\ar@{->}_{'\sfq^-} "V";"W"};
\endxy
\end{equation} 
Our goal is to construct a canonical morphism from $\on{Id}_{\Dmod(\CZ)}$ to the composed functor
\begin{equation} \label{e:other comp}
\left((\sfp^-)_\bullet\circ (\sfq^-)^!\right)\circ \left((\sfq^+)_\bullet \circ (\sfp^+)^!\right):\Dmod(\CZ)\to \Dmod(\CZ).
\end{equation}
The good news is that all morphisms in diagram \eqref{e:unit diagram} are representable.
In particular, $\sfp^-$ and $\sfq^+$ are representable, so $(\sfp^-)_\bullet=(\sfp^-)_\blacktriangle$ and 
$(\sfq^+)_\bullet =(\sfq^+)_\blacktriangle\,$. 

\medskip

Thus, the problem is to construct a canonical morphism from 
$\on{Id}_{\Dmod(\CZ)}$ to the  composed functor
\begin{equation} \label{e:non-dangerous}
\left((\sfp^-)_\blacktriangle\circ (\sfq^-)^!\right)\circ \left((\sfq^+)_\blacktriangle \circ (\sfp^+)^!\right):\Dmod(\CZ)\to \Dmod(\CZ).
\end{equation}
Using base change\footnote{Since we have switched to the renormalized direct images, we can apply base change and do other standard manipulations.}, we further identify the functor \eqref{e:non-dangerous} with
\begin{equation} \label{e:other comp triangle}
(\sfp^-\circ {}'\sfq^+)_\blacktriangle\circ (\sfp^+\circ {}'\sfq^-)^!,
\end{equation}
where $'\sfq^+$ and $'\sfq^-$ are as in diagram \eqref{e:unit diagram}.

\sssec{}  \label{sss:q0&1}

Set $$\CQ_0:=(\sfp^+\times \sfp^-)_\blacktriangle(\omega_{\CZ^+\underset{\CZ^0}\times \CZ^-})\in \Dmod(\CZ\times \CZ).$$

\medskip

Then the functor \eqref{e:other comp triangle} (and, hence, \eqref{e:other comp}) 
is canonically isomorphic to $\sF_{\CQ_0}\,$. 
 
\medskip

The identity functor $\Dmod(\CZ)\to \Dmod(\CZ)$ equals $\sF_{\CQ_1}$, 
where
$$\CQ_1:=(\Delta_\CZ)_\blacktriangle(\omega_\CZ)\in \Dmod(\CZ\times \CZ).$$

\sssec{}
In \secref{ss:constructing map of kernels} we will construct a canonical map 
\begin{equation} \label{e:map on kernels}
\CQ_1\to \CQ_0\, .
\end{equation}

By Sects. \ref{sss:reformulating}-\ref{sss:q0&1} and \ref{sss:functors and kernels}, the map of kernels 
\eqref{e:map on kernels}
induces a natural transformation
\begin{equation} \label{e:Braden unit}
\on{Id}_{\Dmod(\CZ)}\to \left((\sfp^-)_\blacktriangle\circ (\sfq^-)^!\right)\circ \left((\sfq^+)_\blacktriangle \circ (\sfp^+)^!\right)
\end{equation}
between the corresponding functors.

\medskip

In \secref{s:Verifying} we will prove that the natural transformations \eqref{e:Braden unit}
and \eqref{e:Braden co-unit equiv} satisfy the properties of unit and co-unit of an adjunction between the functors 
$(\sfq^+)_\blacktriangle \circ (\sfp^+)^!$ and $(\sfp^-)_\blacktriangle\circ (\sfq^-)^!$.

\ssec{Constructing the morphism \eqref{e:map on kernels}}   \label{ss:constructing map of kernels}
We will first define an object $$\CQ\in\Dmod(\BA^1\times \CZ\times \CZ)^{\BG_m\on{-mon}}\, ,$$ which 
``interpolates" between $\CQ_1$ and $\CQ_0$. We will then define \eqref{e:map on kernels} to be the 
specialization morphism $\on{Sp}_\CQ\,$.

\sssec{}   \label{sss:concrete situation}

Recall the algebraic space $\wt{Z}$ from \secref{s:deg} and 
set $$\wt\CZ:=\wt{Z}/\BG_m,\quad  \wt\CZ_t:=\wt{Z}_t/\BG_m\simeq \wt\CZ\underset{\BA^1}\times \{t\}$$
(the action of $\BG_m$ on $\wt\CZ$ was defined in \secref{sss:anti-diagonal}).

\medskip

Consider the morphisms
$$\wt\sfp :\wt\CZ\to \BA^1\times \CZ\times \CZ \,\,\,\text{   and   } \,\,\,\wt\sfp_t :\wt\CZ_t\to \CZ\times \CZ$$
induced by the maps \eqref{e:tilde p} and \eqref{e:tilde p_t}, respectively.

\medskip

Set 
$$\CQ:=\wt\sfp_\blacktriangle(\omega_{\wt\CZ})\in \Dmod(\BA^1\times \CZ\times \CZ).$$

\sssec{}  \label{sss:Q mon}

We claim that $\CQ$ belongs to the subcategory $\Dmod(\BA^1\times \CZ\times \CZ)^{\BG_m\on{-mon}}$. 

\medskip

In fact, we claim that $\CQ$ is the pullback of a canonically defined object of the category $\Dmod(\BA^1/\BG_m\times \CZ\times \CZ)$.
Indeed, this follows from the existence of the Cartesian diagram

\medskip

$$
\CD
\wt\CZ   @>{=}>>  \wt{Z}/\BG_m    @>>>   \wt{Z}/\BG_m\times \BG_m   \\
@V{\wt\sfp}VV     @V{\wt{p}/\BG_m}VV     @VVV         \\
\BA^1\times \CZ \times \CZ  @>{=}>> \BA^1\times Z/\BG_m\times Z/\BG_m   @>>>  \BA^1/\BG_m\times Z/\BG_m\times Z/\BG_m,
\endCD
$$
where $\BG_m\times \BG_m$ acts on $\wt{Z}$ as in \secref{sss:action of G_m^2}.

\sssec{}  \label{sss:concrete}

Recall that the pair $(\wt{Z}_1,\wt{p}_1)$ identifies with $(Z,\Delta_Z)$, and the pair
$(\wt{Z}_0,\wt{p}_0)$ identifies with $(Z^+\underset{Z^0}\times Z^-,p^+\times p^-)$. 

\medskip

Therefore, the pair $(\wt\CZ_1,\wt\sfp_1)$ identifies with $(\CZ,\Delta_\CZ)$, and the pair
$(\wt\CZ_0,\wt\sfp_0)$ identifies with $(\CZ^+\underset{\CZ^0}\times \CZ^-,\sfp^+\times \sfp^-)$. 

\medskip

Hence, by base change, the objects $\CQ_1$ and $\CQ_0$ from \secref{sss:q0&1} identify with the !-restrictions of
$\CQ$ to
$$\{1\}\times \CZ\times \CZ\to \BA^1\times \CZ\times \CZ \text{ and }
\{0\}\times \CZ\times \CZ\to \BA^1\times \CZ\times \CZ,$$
respectively.

\medskip

Now, the sought-for map \eqref{e:map on kernels} is given by the map $\on{Sp}_\CQ$ 
of \eqref{e:specialization}.

\section{Verifying the adjunction properties}   \label{s:Verifying}
In \secref{sss:Consider now} we introduced the stacks
$$\CZ:=Z/\BG_m\, ,\,\, \CZ^0:=Z^0/\BG_m\, ,\,\, \CZ^{\pm}:=Z^{\pm}/\BG_m$$
and the morphisms
$\sfp^{\pm}:\CZ^{\pm}\to \CZ\,$, $\sfq^{\pm}:\CZ^{\pm}\to\CZ^0.$ 

\medskip

In 
Sects. \ref{ss:Reformulation of Braden}-\ref{ss:equivariant version}
we constructed a natural transformation
$$\left((\sfq^+)_\bullet \circ (\sfp^+)^!\right)\circ \left((\sfp^-)_\bullet\circ (\sfq^-)^!\right)\to \on{Id}_{\Dmod(\CZ^0)}
\, ,$$
see formula \eqref{e:Braden co-unit equiv}. Since the morphisms $\sfp^-$ and $\sfq^+$ are representable we have 
$(\sfp^-)_\bullet=(\sfp^-)_\blacktriangle$ and 
$(\sfq^+)_\bullet =(\sfq^+)_\blacktriangle\,$.  So the above natural transformation \eqref{e:Braden co-unit equiv} 
can be rewritten as a natural transformation
\begin{equation} \label{e:2Braden co-unit} 
\left((\sfq^+)_\blacktriangle \circ (\sfp^+)^!\right)\circ \left((\sfp^-)_\blacktriangle\circ (\sfq^-)^!\right)\to \on{Id}_{\Dmod(\CZ^0)}\, .
\end{equation}

In \secref{s:unit} we  constructed a natural transformation
\begin{equation}   \label{e:2Braden unit}
 \on{Id}_{\Dmod(\CZ)}\to \left((\sfp^-)_\blacktriangle\circ (\sfq^-)^!\right)\circ \left((\sfq^+)_\blacktriangle \circ (\sfp^+)^!\right),
\end{equation}
see formula \eqref{e:Braden unit}.

\medskip

To prove Theorem~\ref{t:Braden adj equiv}, it suffices to
show that the compositions
\begin{equation} \label{e:first composition}
(\sfp^-)_\blacktriangle\circ (\sfq^-)^!\to \left((\sfp^-)_\blacktriangle\circ (\sfq^-)^!\right)\circ \left((\sfq^+)_\blacktriangle \circ 
(\sfp^+)^!\right)\circ \left((\sfp^-)_\blacktriangle\circ (\sfq^-)^!\right)\to (\sfp^-)_\blacktriangle\circ (\sfq^-)^!
\end{equation}
and 
\begin{equation} \label{e:second composition}
(\sfq^+)_\blacktriangle \circ (\sfp^+)^!\to \left((\sfq^+)_\blacktriangle \circ 
(\sfp^+)^!\right)\circ \left((\sfp^-)_\blacktriangle\circ (\sfq^-)^!\right)\circ \left((\sfq^+)_\blacktriangle \circ (\sfp^+)^!\right)\to
(\sfq^+)_\blacktriangle \circ (\sfp^+)^!\
\end{equation}
corresponding to \eqref{e:2Braden co-unit} and \eqref{e:2Braden unit} are isomorphic to \footnote{In the future we will 
skip the words ``isomorphic to"  in similar situations. (This is a slight abuse of language since we work with the DG categories of 
D-modules rather than with their homotopy categories.)} the identity morphisms.

\medskip

We will do so for the composition \eqref{e:first composition}. The case of \eqref{e:second composition}
is similar and will be left to the reader. 

\medskip

The key point of the proof is \secref{sss:key}, which relies on the geometric \propref{p:2open embeddings}.
More precisely, we use the part of \propref{p:2open embeddings} about $Z^-$.
To treat the composition \eqref{e:second composition}, one has to use the part of \propref{p:2open embeddings} about $Z^+$. 

\ssec{The diagram describing the composed functor}

\sssec{The big diagram}

We will use the notation
\begin{equation}   \label{e:BIG}
\BIG:=\left((\sfp^-)_\blacktriangle\circ (\sfq^-)^!\right)\circ \left((\sfq^+)_\blacktriangle \circ (\sfp^+)^!\right)\circ
\left((\sfp^-)_\blacktriangle\circ (\sfq^-)^!\right).
\end{equation}

\medskip

By base change, $\BIG$ is given by pull-push along the following diagram: 

\begin{equation} \label{e:comp diag 1}
\xy
(-20,0)*+{\CZ^0}="X";
(20,0)*+{\CZ\, .}="Y";
(0,20)*+{\CZ^-}="Z";
(-40,20)*+{\CZ^+}="W";
(-60,0)*+{\CZ}="U";
(-20,40)*+{\CZ^+\underset{\CZ^0}\times \CZ^-}="V";
(-100,0)*+{\CZ^0}="T";
(-80,20)*+{\CZ^-}="S";
(-60,40)*+{\CZ^-\underset{\CZ}\times \CZ^+}="R";
(-40,60)*+{\CZ^-\underset{\CZ}\times \CZ^+\underset{\CZ^0}\times \CZ^-}="Q";
{\ar@{->}_{\sfq^-} "Z";"X"};
{\ar@{->}^{\sfp^-} "Z";"Y"};
{\ar@{->}^{\sfq^+} "W";"X"};
{\ar@{->}_{\sfp^+} "W";"U"}; 
{\ar@{->} "V";"Z"};
{\ar@{->} "V";"W"};
{\ar@{->}_{\sfq^-} "S";"T"};
{\ar@{->}^{\sfp^-} "S";"U"};
{\ar@{->} "R";"S"};
{\ar@{->} "R";"W"};
{\ar@{->} "Q";"R"};
{\ar@{->} "Q";"V"};
\endxy
\end{equation}

\sssec{Some notation}

Set
\begin{equation}   \label{e:tilde Z-}
\wt{Z}^-:=Z^-\underset{Z}\times \wt{Z}\, ,
\end{equation}
where the fiber product is formed using the composition
\[
\wt{Z} \overset{\wt{p}}\longrightarrow \BA^1\times Z\times Z\to Z\times Z \overset{\on{pr}_1}\longrightarrow Z
\]
(i.e., the morphism $\pi_1:\wt{Z}\to Z$ from \secref{sss:tilde p}).

\medskip

For $t\in \BA^1$ set $\wt{Z}^-_t:=Z^-\underset{Z}\times \wt{Z}_t\,$.

\medskip

Let $\wt{p}^-:\wt{Z}^-\to \BA^1\times Z^-\times Z$ denote  the map obtained by base change from
$$\wt{p}:\wt{Z} \to \BA^1\times Z\times Z\,.$$
Let 
$$r: \wt{Z}^-\to \BA^1\times Z^0\times Z$$
denote the composition of $\wt{p}^-:\wt{Z}^-\to \BA^1\times Z^-\times Z$ with the morphism
$$\id_{\BA^1}\times q^-\times \id_Z:\BA^1\times Z^-\times Z\to\BA^1\times Z^0\times Z.$$
Let $r_t:\wt{Z}^-_t\to Z^0\times Z$ denote the morphism induced by $r: \wt{Z}^-\to \BA^1\times Z^0\times Z\,$.

\sssec{More notation}   \label{sss:stacky notation}

Recall that
$$\wt\CZ:=\wt{Z}/\BG_m,\quad  \wt\CZ_t:=\wt{Z}_t/\BG_m\simeq \wt\CZ\underset{\BA^1}\times \{t\}\, ,$$
where $\wt{Z}$ is the algebraic space from \secref{s:deg}
(the action of $\BG_m$ on $\wt\CZ$ was defined in \secref{sss:anti-diagonal}).

\medskip

Set
$$\wt\CZ^-:=\CZ^-\underset{\CZ}\times \wt{\CZ}=\wt Z^-/\BG_m\, ,\quad\quad  
\wt\CZ^-_t:=\CZ^-\underset{\CZ}\times \wt{\CZ}_t=\wt Z^-_t/\BG_m\,.$$

Let 
$$\sfr:\wt\CZ^-\to \BA^1\times \CZ^0\times \CZ\, ,\quad \sfr_t:\wt\CZ^-_t\to \CZ^0\times \CZ$$
be the morphisms induced by $r: \wt{Z}^-\to \BA^1\times Z^0\times Z$ and $r_t:\wt{Z}^-_0\to Z^0\times Z$, respectively. In particular, we have the morphisms $\sfr_0$ and $\sfr_1$ corresponding to $t=0$ and $t=1$.

\sssec{A smaller diagram describing the functor $\BIG$}  \label{sss:small diagram}
By \propref{p:tilde Z_0}, we have an isomorphism $\wt{Z}_0\iso Z^+\underset{Z^0}\times Z^-$.
The corresponding isomorphism 
$$\wt{Z}^-_0:=Z^-\underset{Z}\times \wt{Z}_0\iso Z^-\underset{Z}\times Z^+\underset{Z^0}\times Z^-$$ 
induces
an isomorphism
$$\wt\CZ^-_0\simeq \CZ^-\underset{\CZ}\times \CZ^+\underset{\CZ^0}\times \CZ^-.$$

Thus the upper term of diagram \eqref{e:comp diag 1} is $\wt\CZ^-_0$. The compositions 
\[
\CZ^-\underset{\CZ}\times \CZ^+\underset{\CZ^0}\times \CZ^-\to \CZ^-\underset{\CZ}\times \CZ^+\to 
\CZ^-\overset{\, \sfq^-}\longrightarrow \CZ^0 \quad \mbox{and}\quad
\CZ^-\underset{\CZ}\times \CZ^+\underset{\CZ^0}\times \CZ^-\to \CZ^+\underset{\CZ^0}\times \CZ^-\to \CZ^-
\overset{\, \sfp^-}\longrightarrow  \CZ
\]
from diagram \eqref{e:comp diag 1} are equal, respectively, to the compositions
\[
\wt{\CZ}^-_0\overset{\sfr_0}\longrightarrow \CZ^0\times \CZ \overset{\on{pr}_1}\longrightarrow \CZ^0
\quad \mbox{and}\quad
\wt{\CZ}^-_0\overset{\sfr_0}\longrightarrow \CZ^0\times \CZ \overset{\on{pr}_2}\longrightarrow \CZ
\]
(the morphism $\sfr_0$ was defined in \secref{sss:stacky notation}).

\medskip

Hence, the functor $\BIG$ is given by pull-push along the diagram

\begin{equation}  \label{e:comp diag B}
\xy
(-20,0)*+{\CZ^0}="X";
(20,0)*+{\CZ\, .}="Y";
(0,20)*+{\wt\CZ^-_0}="Z";
{\ar@{->}_{\on{pr}_1\circ \sfr_0}  "Z";"X"};
{\ar@{->}^{\on{pr}_2\circ \sfr_0}  "Z";"Y"};
\endxy
\end{equation}

\ssec{The natural transformations at the level of kernels}  \label{ss:nat trans via kernels}
The goal of this subsection is to describe the natural transformations 
$$\BIG\to (\sfp^-)_\blacktriangle\circ (\sfq^-)^! \text{ and } (\sfp^-)_\blacktriangle\circ (\sfq^-)^! \to \BIG$$
at the level of kernels.

\sssec{The kernel corresponding to $\BIG$}

Set
$$\CS:=\sfr_\blacktriangle (\omega_{\wt\CZ^-})\in \Dmod(\BA^1\times \CZ^0\times \CZ),$$
where $\sfr:\wt\CZ^-\to \BA^1\times \CZ^0\times \CZ$ was defined in \secref{sss:stacky notation}.

\medskip

As in \secref{sss:Q mon}, one shows that
$$\CS\in \Dmod(\BA^1\times \CZ^0\times \CZ)^{\BG_m\on{-mon}}\, .$$

Set also
$$\CS_0:=(\sfr_0)_\blacktriangle (\omega_{\wt\CZ^-_0})\in \Dmod(\CZ^0\times \CZ),\quad \quad
\CS_1:=(\sfr_1)_\blacktriangle (\omega_{\wt\CZ^-_1})\in \Dmod(\CZ^0\times \CZ).$$
By \secref{sss:small diagram}, the functor $\Phi$ identifies with $\sF_{\CS_0}\,$. 

\sssec{The kernel corresponding to $(\sfp^-)_\blacktriangle\circ (\sfq^-)^! $}
Now set
$$\CT:=(\sfq^-\times \sfp^-)_\blacktriangle(\omega_{\CZ^-}).$$ 
We have 
$$(\sfp^-)_\blacktriangle\circ (\sfq^-)^! \simeq \sF_{\CT}\, .$$

\sssec{}  \label{sss:j tilde}

Recall the open embedding
$$j:Z^0\hookrightarrow Z^-\underset{Z}\times Z^+,$$
see \propref{p:Cartesian}.  

\medskip

Let $j^-$ denote the corresponding open embedding 
$$Z^-\hookrightarrow Z^-\underset{Z}\times Z^+\underset{Z^0}\times Z^-\simeq Z^-\underset{Z}\times \wt{Z}_0=: \wt{Z}^-_0,$$
obtained by base change. 

\medskip

Let $\sfj^-$ denote the corresponding open embedding
$$\CZ^-\hookrightarrow \wt\CZ^-_0.$$

Note that the composition
\[
\CZ^-\overset{\,\sfj^-}\hookrightarrow \wt\CZ^-_0\overset{\sfr_0}\longrightarrow \CZ^0\times \CZ
\]
equals $\sfq^-\times \sfp^-$.

\sssec{The morphism $\BIG\to(\sfp^-)_\blacktriangle\circ (\sfq^-)^!$ at the level of kernels}
\label{sss:1at the level of kernels} \hfill

\medskip

Recall that the morphism $\BIG\to(\sfp^-)_\blacktriangle\circ (\sfq^-)^!$ comes from the morphism 
$$\left((\sfq^+)_\blacktriangle \circ (\sfp^+)^!\right)\circ \left((\sfp^-)_\blacktriangle\circ (\sfq^-)^!\right)\to \on{Id}_{\Dmod(\CZ^0)}$$ 
constructed in Sects. \ref{ss:Reformulation of Braden}-\ref{ss:equivariant version}.
By construction,  the natural transformation $$\BIG\to(\sfp^-)_\blacktriangle\circ (\sfq^-)^!$$ 
corresponds to the map of kernels
\begin{equation} \label{e:S_0 to T}
\CS_0\to \CT
\end{equation}
equal to the composition
$$\CS_0:=(\sfr_0)_\blacktriangle (\omega_{\wt\CZ^-_0})\to
(\sfr_0)_\blacktriangle\circ \sfj^-_\blacktriangle (\omega_{\CZ^-}) \iso (\sfq^-\times \sfp^-)_\blacktriangle(\omega_{\CZ^-})=:\CT,$$
where the first arrow comes from
$$\omega_{\wt\CZ^-_0}\to \sfj^-_\bullet\circ (\sfj^-)^\bullet(\omega_{\wt\CZ^-_0})\simeq \sfj^-_\bullet(\omega_{\CZ^-})\simeq
\sfj^-_\blacktriangle(\omega_{\CZ^-}).$$

\sssec{The isomorphism $\CT\simeq \CS_1\,$} \hfill

\medskip

The (tautological) identification $\wt{Z}_1\simeq Z$ defines an identification
\begin{equation}   \label{e:tautological identification}
\wt\CZ^-_1\simeq \CZ^-,
\end{equation}
so that
the morphism $\sfr_1:\CZ^-_1\rightarrow \CZ^0\times \CZ$  identifies with $\sfq^-\times \sfp^-$.

\medskip

Hence, we obtain a tautological identification
\begin{equation} \label{e:T to S_1}
\CT\simeq \CS_1\, .
\end{equation}

\sssec{The morphism $(\sfp^-)_\blacktriangle\circ (\sfq^-)^!\to\BIG$ at the level of kernels}
\label{sss:2at the level of kernels}  \hfill

\medskip

The map $\on{Sp}_\CS$ of \eqref{e:specialization} defines a canonical map
\begin{equation} \label{e:S_1 to S_0}
\CS_1\to \CS_0\, .
\end{equation}

By \secref{sss:functoriality of specialization}, the natural transformation 
$(\sfp^-)_\blacktriangle\circ (\sfq^-)^!\to\BIG$
comes from the map
\begin{equation} \label{e:T to S_0}
\CT\to \CS_1\to \CS_0\, ,
\end{equation}
equal to the composition of \eqref{e:T to S_1} and \eqref{e:S_1 to S_0}.

\sssec{Conclusion}
Thus, in order to prove that the composition \eqref{e:first composition} is the identity map, 
it suffices to show that the composed map
\begin{equation} \label{e:composed kernels}
\CT\to \CS_1\to \CS_0\to \CT
\end{equation}
is the identity map on $\CT$. 

\ssec{Passing to an open substack}

\sssec{}

Recall the open embedding
$$j^-:Z^-\hookrightarrow \wt{Z}^-_0$$ 
introduced in \secref{sss:j tilde}. 

\medskip

Let $\overset{\circ}{\wt{Z}}{}^-$ denote the open subset of $\wt{Z}^-$ obtained by removing the closed
subset
$$\left(\wt{Z}^-_0-Z^-\right)\subset \wt{Z}^-_0\subset \wt{Z}^-.$$

\medskip

Let $\overset{\circ}{\wt\CZ}{}^-$ denote the corresponding open substack of $\wt\CZ^-$. 
Let $\overset{\circ}{\wt\CZ}{}^-_t$ denote the fiber of $\overset{\circ}{\wt\CZ}{}^-$ over $t\in \BA^1$.

\medskip

By definition, the open embedding 
$$\sfj^-:\CZ^-\hookrightarrow \wt\CZ^-_0$$
defines an \emph{isomorphism}
\begin{equation} \label{e:fiber at 0 open}
\CZ^-\iso \overset{\circ}{\wt\CZ}{}^-_0.
\end{equation}

\medskip

Note that the isomorphism $\CZ^-\iso \wt\CZ{}^-_1$ of \eqref{e:tautological identification} still defines an isomorphism
\begin{equation} \label{e:fiber at 1 open}
\CZ^-\iso \overset{\circ}{\wt\CZ}{}^-_1.
\end{equation}

\sssec{}

Let 
$$\osfr:\overset{\circ}{\wt\CZ}{}^-\to \BA^1\times \CZ^0\times \CZ\quad  \text{ and } \quad
\osfr_t:\overset{\circ}{\wt\CZ}{}^-\to \CZ^0\times \CZ$$
denote the morphisms induced by the maps $\sfr$ and $\sfr_t$ from
\secref{sss:stacky notation}.

\medskip

Set
$$\oCS:=\osfr_\blacktriangle(\omega_{\overset{\circ}{\wt\CZ}{}^-}),$$
and also
$$\oCS_0:=(\osfr_0)_\blacktriangle(\omega_{\overset{\circ}{\wt\CZ}{}^-_0})\quad \text{ and } \quad
\oCS_1:=(\osfr_1)_\blacktriangle(\omega_{\overset{\circ}{\wt\CZ}{}^-_1}).$$

\medskip

The open embedding $\overset{\circ}{\wt\CZ}{}^-\hookrightarrow \wt\CZ^-$ gives rise to the maps
$$\CS\to \oCS,\quad  \CS_0\to \oCS_0, \quad \CS_1\to \oCS_1\, .$$

\medskip

As in 
Sects. \ref{sss:1at the level of kernels}-\ref{sss:2at the level of kernels}, we have the natural transformations
\begin{equation}  \label{e:composed kernels open}
\CT\to \oCS_1\to \oCS_0\to \CT.
\end{equation}

Moreover, the diagram 
$$
\CD
\CT   @>>>  \CS_1  @>>> \CS_0  @>>>  \CT  \\
@V{\id}VV  @VVV    @VVV   @VV{\id}V   \\
\CT  @>>>  \oCS_1   @>>>   \oCS_0  @>>>  \CT
\endCD
$$
commutes. 

\medskip

Hence, 
in order to show that the composed map \eqref{e:composed kernels} is the identity map,
\emph{it suffices to show that the composed map \eqref{e:composed kernels open} is the identity map. }
We will do this in the next subsection.

\ssec{The key argument}   

\sssec{}   \label{sss:key}

Recall now the open embedding
$$\BA^1\times Z^-\to Z^-\underset{Z}\times \wt{Z}=:\wt{Z}^-$$
of \eqref{e:embedding2}.  

\medskip

By definition, it induces an isomorphism 
$$\BA^1\times Z^-\simeq \overset{\circ}{\wt{Z}}{}^-.$$

Dividing by the action of $\BG_m$, we obtain an isomorphism
\begin{equation} \label{e:key}
\BA^1\times \CZ^-\simeq \overset{\circ}{\wt\CZ}{}^-.
\end{equation}

Under this identification, we have:

\medskip

\begin{itemize}

\item
The map $\osfr:\overset{\circ}{\wt\CZ}{}^-\to \BA^1\times \CZ^0\times \CZ$
identifies with the map $\BA^1\times \CZ^-\to\BA^1\times \CZ^0\times \CZ$ induced by
$\id_{\BA^1}:\BA^1\to \BA^1$ and $ (\sfq^-\times \sfp^-): \CZ^-\to \CZ^0\times \CZ\,$.

\medskip

\item
The isomorphism $\CZ^-\iso \overset{\circ}{\wt\CZ}{}^-_1$ of \eqref{e:fiber at 1 open}
corresponds to the identity map
$$\CZ^-\to (\BA^1\times \CZ^-)\underset{\BA^1}\times \{1\}\simeq \CZ^-.$$

\item
The isomorphism $\CZ^-\iso \overset{\circ}{\wt\CZ}{}^-_0$ of \eqref{e:fiber at 0 open}
corresponds to the identity map
$$\CZ^-\to (\BA^1\times \CZ^-)\underset{\BA^1}\times \{0\}\simeq \CZ^-.$$

\end{itemize}

\sssec{}

Hence, we obtain that the composition \eqref{e:composed kernels open} identifies with
$$\CT\simeq \iota_1^!(\omega_{\BA^1}\boxtimes \CT) \overset{\on{Sp}}\longrightarrow \iota_0^!(\omega_{\BA^1}\boxtimes \CT) \simeq \CT,$$
where $\on{Sp}:=\on{Sp}_{\omega_{\BA^1}\boxtimes \CT}$
is the specialization map \eqref{e:specialization} for the object
$$\omega_{\BA^1}\boxtimes \CT \in \Dmod(\BA^1\times \CZ^0\times \CZ).$$

\medskip

The fact that the above map is the identity map on $\CT$ follows from \secref{sss:specialization for constant}.

\appendix

\section{Pro-categories}  \label{s:pro} \hfill

\medskip

\noindent{\bf A.1.} For a DG category $\bC$ let $\on{Pro}(\bC)$ denote its pro-completion, thought of as the DG category opposite to
that of covariant exact functors $\bC\to \Vect$, where $\Vect$ denotes the DG category of complexes of $k$-vector spaces.
\footnote{A way to deal with set-theoretical difficulties is to require
that our functors commute with $\kappa$-filtered colimits for some cardinal $\kappa$, see 
\cite[Def.~5.3.1.7]{Lur}.} 

\medskip

Yoneda embedding defines
a fully faithful functor $\bC\to \on{Pro}(\bC)$. Any object in $\on{Pro}(\bC)$ can be written as a filtered limit
(taken in $\on{Pro}(\bC)$) of co-representable functors. 

\medskip

\noindent{\bf A.2.} 
A functor $\sF:\bC'\to \bC''$
between DG categories induces a functor denoted also by $\sF$
$$\on{Pro}(\bC')\to \on{Pro}(\bC'')$$ by applying the 
\emph{right Kan extension} of the functor
$$\bC'\overset{\sF}\longrightarrow \bC''\hookrightarrow \on{Pro}(\bC'')$$ 
along the embedding $\bC'\to \on{Pro}(\bC')$. 

\medskip

The same construction can be phrased as follows: for $\wt\bc'\in \on{Pro}(\bC')$, thought
of as a functor $\bC'\to \Vect$, the object $\sF(\wt\bc')$, thought of as a functor $\bC''\to \Vect$, is the 
\emph{left Kan extension} of $\wt\bc'$ along the functor $\sF:\bC'\to \bC''$. 

\medskip

Explicitly, if $\wt\bc\in \on{Pro}(\bC')$ is written as $\underset{i\in I}{\underset{\longleftarrow}{lim}}\, \bc_i$
with $\bc_i\in \bC'$, then
$$\sF(\wt\bc)\simeq \underset{i\in I}{\underset{\longleftarrow}{lim}}\, \sF(\bc_i),$$
as objects of $\on{Pro}(\bC'')$ and 
$$\sF(\wt\bc)\simeq \underset{i\in I}{\underset{\longrightarrow}{lim}}\, \CMaps_{\bC''}(\sF(\bc_i),-),$$
as functors $\bC''\to \Vect$. 

\medskip

\noindent{\bf A.3.}  
Let $\sG:\bC'\to \bC''$ be a functor between DG categories. We can speak of its left adjoint $\sG^L$
as a functor $\bC''\to \on{Pro}(\bC')$. Namely, for $\bc''\in \bC''$ the object $\sG^L(\bc'')\in \on{Pro}(\bC')$,
thought of as a functor $\bC'\to \Vect$ is given by
$$(\sG^L(\bc''))(\bc')=\CMaps_{\bC''}(\bc'',\sG(\bc')).$$

\medskip

\noindent{\bf A.4.} We let the same symbol $\sG^L$ also denote the functor $\on{Pro}(\bC'')\to  \on{Pro}(\bC')$ obtained as the right Kan extension of 
$\sG^L:\bC''\to  \on{Pro}(\bC')$ along
$\bC''\hookrightarrow  \on{Pro}(\bC'')$. 

\medskip

The functor $\sG^L$ is the left adjoint of the functor 
$\sG:\on{Pro}(\bC')\to \on{Pro}(\bC'')$. 

\medskip

We can also think of $\sG^L$ as follows: for
$\wt\bc''\in \on{Pro}(\bC'')$, thought of as a functor $\bC''\to \Vect$, the object $\sG^L(\wt\bc'')$, thought of as a functor
$\bC'\to \Vect$ is given by
$$(\sG^L(\wt\bc''))(\bc')=\wt\bc''(\sG(\bc')).$$

\end{document}